\documentclass[a4paper,leqno,12pt,twoside]{amsart}
\usepackage{amsmath,amsfonts,amssymb,amsthm,amscd}
\usepackage[utf8]{inputenc}
\usepackage[T1]{fontenc}
\usepackage[english]{babel}
\usepackage[colorlinks=true,citecolor=blue, urlcolor=blue, linkcolor=blue,pagebackref]{hyperref}
\usepackage[top=1.2in,bottom=1.2in,left=1.in,right=1.in]{geometry} 
\usepackage{graphicx}
\usepackage{paralist}
\usepackage{tabto}
\usepackage{standalone}
\usepackage{tikz}
\usetikzlibrary{matrix}
\usetikzlibrary{arrows}
\usepackage{xfrac}
\usepackage{amsthm, amsmath, amssymb,latexsym}
\usepackage{tikz-cd}
\usepackage[alphabetic,backrefs,lite]{amsrefs}
\usepackage[all]{xy} 

\usepackage{enumerate}
\usepackage{xcolor}
\usepackage{aliascnt}

\newtheorem{theorem}{Theorem}[section]
\newtheorem{lemma}[theorem]{Lemma}
\newtheorem{corollary}[theorem]{Corollary}
\newtheorem{proposition}[theorem]{Proposition}
\newtheorem{remark}[theorem]{Remark}
\newtheorem{definition}[theorem]{Definition}

\newcommand{\nc}{\newcommand} 
\nc{\cH}{{\mathcal H}}
\nc{\cA}{{\mathcal A}}
\nc{\cG}{{\mathcal G}}
\nc{\cC}{{\mathcal C}}
\nc{\cO}{{\mathcal O}}
\nc{\cI}{{\mathcal I}}
\nc{\cB}{{\mathcal B}}
\nc{\cY}{{\mathcal Y}}
\nc{\cK}{{\mathcal K}} 
\nc{\cX}{{\mathcal X}}
\nc{\cS}{{\mathcal S}}
\nc{\cE}{{\mathcal E}}
\nc{\cF}{{\mathcal F}}
\nc{\cZ}{{\mathcal Z}}
\nc{\cQ}{{\mathcal Q}}
\nc{\cN}{{\mathcal N}}
\nc{\cP}{{\mathcal P}}
\nc{\cL}{{\mathcal L}}
\nc{\cM}{{\mathcal M}}
\nc{\cT}{{\mathcal T}}
\nc{\cW}{{\mathcal W}}
\nc{\cU}{{\mathcal U}}
\nc{\cJ}{{\mathcal J}}
\nc{\cV}{{\mathcal V}}
\nc{\bH}{{\mathbb H}}
\nc{\bA}{{\mathbb A}}
\nc{\bG}{{\mathbb G}}
\nc{\bC}{{\mathbb C}}
\nc{\bO}{{\mathbb O}}
\nc{\bI}{{\mathbb I}}
\nc{\bB}{{\mathbb B}}
\nc{\bY}{{\mathbb Y}}
\nc{\bK}{{\mathbb K}} 
\nc{\bX}{{\mathbb X}}
\nc{\bS}{{\mathbb S}}
\nc{\bE}{{\mathbb E}}
\nc{\bF}{{\mathbb F}}
\nc{\bZ}{{\mathbb Z}}
\nc{\bQ}{{\mathbb Q}}
\nc{\bN}{{\mathbb N}}
\nc{\bP}{{\mathbb P}}
\nc{\bL}{{\mathbb L}}
\nc{\bM}{{\mathbb M}}
\nc{\bT}{{\mathbb T}}
\nc{\bW}{{\mathbb W}}
\nc{\bU}{{\mathbb U}}
\nc{\bD}{{\mathbb D}}
\nc{\bJ}{{\mathbb J}}
\nc{\bV}{{\mathbb V}}
\nc{\bbZ}{{\mathbb Z}}
\nc{\bR}{{\mathbb R}}
\nc{\fr}{{\rightarrow}}
\nc{\co}{{\nabla}}
\newcommand{\debar}{{\bar\partial}}

\nc{\cu}{{\barline{\nabla}}}

\title{Asymptotic directions in the moduli space of curves}

\author{E. Colombo} 

\address{Elisabetta Colombo \\ Universit\`a degli Studi di Milano  \\ Dipartimento di Matematica \\ Via Cesare Saldini 50  \\   20133 Milano, Italy }
\email{elisabetta.colombo@unimi.it}

\author{P. Frediani}
\address{Paola Frediani  \\ Universit\`a degli Studi di Pavia  \\ Dipartimento di Matematica \\ Via Ferrata 1  \\ 27100 Pavia, Italy  }
 \email{paola.frediani@unipv.it}

 \author{G.P. Pirola}
 \address{Gian Pietro Pirola  \\ Universit\`a degli Studi di Pavia  \\ Dipartimento di Matematica \\ Via Ferrata 1  \\ 27100 Pavia, Italy  }
 \email{gianpietro.pirola@unipv.it}

\begin{document}

\begin{abstract}
In this paper we study asymptotic directions in the tangent bundle of the moduli space ${\mathcal M}_g$ of curves of genus $g$, namely those tangent directions that are annihilated by the second fundamental form of the Torelli map. We give examples of asymptotic directions for any $g \geq 4$. We prove that if the rank $d$ of a tangent direction $\zeta \in H^1(T_C)$ (with respect to the infinitesimal deformation map) is less than the Clifford index of the curve $C$, then $\zeta$ is not asymptotic. If the rank of $\zeta$ is equal to the Clifford index of the curve, we give sufficient conditions ensuring that the infinitesimal deformation $\zeta$ is not asymptotic. Then we determine all asymptotic directions of rank 1 and we give an almost complete description of asymptotic directions of rank 2. 

\end{abstract}

\thanks{E. Colombo, P. Frediani and G.P. Pirola are members of GNSAGA (INdAM) and are partially supported by PRIN project {\em Moduli spaces and special varieties} (2022).}

\maketitle

\section{Introduction}

In this paper we study the local geometry of the Torelli locus in
${\mathcal A}_g$. Following the philosophy of Griffiths, the local
geometry of the period map often contains information about the global
geometry (see \cite{griffiths1}, \cite{griffiths2}, \cite{griffiths},
\cite{voisin}).  We consider ${\mathcal A}_g$ endowed with the Siegel
metric, that is the orbifold metric induced by the symmetric metric on
the Siegel space ${\mathcal H}_g = Sp(2g, {\mathbb R})/U(g)$ of which
${\mathcal A}_g$ is a quotient by the action of $Sp(2g, {\mathbb
  Z})$. Denote by $j: {\mathcal M}_g \rightarrow {\mathcal A}_g$ the
Torelli map. The Torelli locus is the closure of the image of $j$. The
local geometry of the map $j$ is governed by the second fundamental
form, which  at a non hyperelliptic curve $C$ of genus $g$,  is a
linear map
$$
II: I_2 \rightarrow Sym^2H^0(K_C^{\otimes 2}),
$$
where $I_2$ is the vector space of quadrics containing the canonical
curve.

One of the leading problems in the area is to study totally geodesic subvarieties of ${\mathcal A}_g$ generically contained in the Torelli locus. 

This problem is related to the Coleman-Oort conjecture according to which for $g$ sufficiently high there should not exist special (or Shimura) subvarieties of ${\mathcal A}_g$ generically contained in the Torelli locus. We recall that special subvarieties of ${\mathcal A}_g$ are totally geodesic, hence in the last years the study of the second fundamental form has been used to attack this problem. In particular estimates on the maximal dimension of a totally geodesic subvariety of ${\mathcal A}_g$ generically contained in the Torelli locus have been given in \cite{cfg}, \cite{gpt}, \cite{fp}.

In this paper  we take a different point of view.
The image of $II$ in $Sym^2H^1(T_C)^{\vee}$ is a linear system of quadrics in ${\mathbb P}H^1(T_C) \cong {\mathbb P}^{3g-4} $. This paper is devoted to the study of the base locus  of this linear system of quadrics. 
Following the terminology of differential geometry we call \emph{asymptotic direction}  a nonzero tangent direction $\zeta \in H^1(T_C)$ such that  $II(Q)(\zeta \odot \zeta) =0$ for all $Q \in I_2$. So asymptotic directions correpond to points in the base locus.

Clearly a tangent direction to a totally geodesic subvariety is
asymptotic.  But the locus of asymptotic directions in the projective
tangent bundle of the moduli space of curves ${\mathcal M}_g $ is a
natural locus, that is worthwhile investigating.  For example we
believe that asymptotic directions could also be useful in the study
of fibred surfaces in relation with Xiao conjecture (\cite{Xiao},
\cite{gt}, see Remark \ref{xiao}).

Since $II$ is injective (see \cite[Corollary 3.4]{cf-trans}), its
image in $Sym^2H^1(T_C)^{\vee}$ is a linear system of quadrics in
${\mathbb P}H^1(T_C) \cong {\mathbb P}^{3g-4} $ of dimension
$\frac{(g-2)(g-3)}{2}$.  Hence for every curve $C$ of genus
$g \leq 9$, $dim( II(I_2)) < 3g-4$, so the intersections of the
quadrics in $II(I_2)$ is non empty, thus there exist asymptotic
directions. In fact for $g\leq 7$ there are examples of special
subvarieties of ${\mathcal A}_g$ generically contained in the Torelli
locus for $g \leq 7$ (see \cite{dejong-noot}, \cite{fgp}, \cite{fgs},
\cite{fpp}, \cite{moonen-special}, \cite{moonen-oort}, \cite{rohde},
\cite{shimura}, \cite{spelta}).

On the other hand, for high values of $g$ one would expect that the intersection of a space of  quadrics  of dimension $\frac{(g-2)(g-3)}{2}$ in ${\mathbb P}^{3g-4} $ would be empty. 

One main result of this paper is to show that this is not always the case. Indeed, for all $g$ there are examples of asymptotic directions given by the Schiffer variations at the ramifications points of the $g^1_3$'s for trigonal curves and by linear combinations of two Schiffer variations on bielliptic curves (see Lemma \ref{lem_schiffer} and Theorem \ref{bielliptic}). 

This is rather unexpected and intriguing and indicates that understanding the geometry
of the locus of asymptotic directions is important.  Especially
finding new examples of curves admitting asymptotic directions would
be very interesting.

In this paper, using the Hodge Gaussian maps introduced in \cite{cpt}, we develop a new technique to calculate the second fundamental form $II(Q)(\zeta \odot \zeta)$ on certain tangent directions $\zeta$ different from Schiffer variations, computing some residues of meromorphic forms (see Proposition \ref{w}). 

This technique works for those tangent directions $\zeta$ whose rank
is less than $g$, where the rank of an infinitesimal deformation
$\zeta$ is the rank of the linear map
$$\cup \zeta: H^0(K_C) \rightarrow H^1({\mathcal O}_C)$$
given by the cup product.

One of the main results we obtain by application of these ideas is
the following

\begin{theorem} (Theorem \ref{cliff>d}, Theorem \ref{cliff-rank})
\label{Intro1}
Let $C$ be a smooth curve of genus $g \geq 4$,  take an integer $d < Cliff(C)$ and an infinitesimal deformation $\zeta \in H^1(T_C)$ of rank $d$. 
Then we have: 
\begin{enumerate}

\item $\zeta$ is a linear combination of (possibly higher) Schiffer variations supported on an effective  divisor $D$ of degree $d$. 
\item $\zeta$ is not asymptotic. 
\end{enumerate}
\end{theorem}

For a definition of $n^{th}$-Schiffer variations see section \ref{section2}. 

Notice that the first part of the above result can be seen as  a generalisation of the generic Torelli theorem of Griffiths. In fact, denoting by $C_d$ the symmetric product of $C$, under the assumption $Cliff(C) > d$, we characterise the image of the natural map 
$${\mathbb P}T_{C_d}\to {\mathbb P}(H^1(T_C))$$
as the locus of deformations of rank  at most $d$. When $d=1$ it is the bicanonical curve. 

In particular, if a curve $C$ of genus $g$ has maximal Clifford index, equal to $\lfloor\frac{g-1}{2}\rfloor$, this 
gives a characterisation of the image of the natural map ${\mathbb P}T_{C_d}\to {\mathbb P}(H^1(T_C))$
as the locus of deformations of rank  at most $d$, for all $d< \lfloor\frac{g-1}{2}\rfloor$.  We recall that this applies to the general curve in moduli space, which has maximal Clifford index. 
Moreover, for such values of $d$, Theorem \ref{Intro1}(2), for curves of maximal Clifford index the base locus of the linear system of quadrics $II(I_2)$ in ${\mathbb P}H^1(T_C)$ does not contain any point $[\zeta]$ with  $Rank(\zeta) \leq d$.



In the case where the rank of $\zeta$ is equal to the Clifford index of the curve, we give sufficient conditions ensuring that the infinitesimal deformation $\zeta$ is not asymptotic (see Theorem \ref{cliff=d}). 

This result allows us to determine all asymptotic directions  of rank 1.  Notice that  by Theorem \ref {Intro1}, there can be asymptotic directions of rank 1 only for curves with Clifford index 1, namely trigonal curves or plane quintics. We have the following

\begin{theorem} (Theorem \ref{rank1} and Theorem \ref{quintic})
\label{theorem 1}
\begin{enumerate}
\item If $C$ is trigonal (non hyperelliptic) of genus $g \geq 8$, or of genus $g =6,7$ and Maroni degree $2$,  then  rank one asymptotic directions are exactly the Schiffer variations in the ramification points of the $g^1_3$. 
\item On a smooth plane quintic there are no rank one asymptotic directions. 
\end{enumerate}
\end{theorem}


We recall that the general trigonal curve of genus $g \geq 6$ has Maroni degree 2.
For trigonal curves of genus $g=5$ or $g=6,7$ and Maroni degree 1, we show that there can exist asymptotic directions that are not Schiffer variations in the ramification points of the $g^1_3$. We describe these asymptotic directions and we give the explicit equations of the trigonal curves admitting such asymptotic  directions (see Section \ref{Maroni}).

Finally we consider infinitesimal deformations of rank 2 and we prove the following 
\begin{theorem} (Theorem \ref{rank2}, Theorem \ref{sextic1})
\label{theorem 2}
\begin{enumerate}
\item Assume $C$ is tetragonal, of genus at least 16 and $C$ is not a double cover of a curve of genus 1 or 2.  If a deformation $\zeta$ of rank 2  is not a linear combination of Schiffer variations supported on an effective degree 2 divisor, then $\zeta$ is not asymptotic.  
\item On a smooth plane sextic there are no asymptotic directions of rank 2. 
\end{enumerate}
\end{theorem}

For bielliptic curves we show the following 
\begin{theorem} (see Theorem \ref{bielliptic})
\label{theorem 3}
On any bielliptic curve of genus at least 5 there exist linear combinations of two Schiffer variations that are asymptotic  of rank 2. 

\end{theorem}

More precisely, the linear combinations of Schiffer variations which are asymptotic in the above Theorem are $\xi_p \pm i \xi_{\sigma(p)}$, where $\sigma$ is the bielliptic involution and $(p, \sigma(p)) \in C \times C$ is in the zero locus of the meromorphic  form $\hat{\eta} \in H^0(K_{C\times C}(2 \Delta))$ which determines the second fundamental form $II$ (see \cite[Theorem 3.7]{cfg}). 

The structure of the paper is as follows. 

In section 2 we describe in Dolbeault cohomology those infinitesimal deformations $\zeta$ whose kernel contains a given nonzero form $\omega \in H^0(K_C)$. We define  a split deformation $\zeta$ as an infinitesimal deformation such that the rank 2 vector bundle of the extension corresponding to $\zeta$ splits as the sum of two line bundles, and we give a Dolbeault cohomology description of it. We also recall the definition of (higher order) Schiffer variations. 

Section 3 contains some technical results that will be useful to make explicit computations on the second fundamental form of the Torelli map.  

In Section 4 we recall the definition of the second fundamental form, we define asymptotic directions and we use the results in Section 3 to give a formula that computes the second fundamental form  $II(Q)(\zeta \odot \zeta)$ on some $\zeta$ with non trivial kernel, in terms of residues of meromorphic forms (see Proposition \ref{w}). 

In Section 5 we consider infinitesimal deformations of rank $d$ less than the Clifford index of the curve, and we analyse the extension corresponding to $\zeta$ and the rank 2 vector bundle $E$ of the extension.  First we show that if either $d < Cliff(C)$, or $d = Cliff(C) < \frac{g-1}{2}$ and $E$  is not globally generated,  then $\zeta$ is a linear combination of Schiffer variations supported on a divisor $D$ of degree $d$ (see Theorem \ref{cliff-rank}). One of the main technical tool is a Theorem of Segre-Nagata and Ghione (see \cite{LazII} p. 84)  on the existence of a subline bundle $A$ of $E$ such that $deg(A) \geq \frac{g-1}{2}$. 
Then we show that if $d < Cliff(C)$, no  linear combinations of Schiffer variations of rank $d$  is  asymptotic (see Theorem \ref{no-schiffer-cliff}). Here we use a result of Green and Lazarsfeld  (\cite[Theorem 1]{GL1}) that allows us to find some quadrics where our technique to compute the second fundamental form works well. Finally we prove Theorem \ref{cliff>d}. 

In Section 6 we define some special split deformations that we call double-split, where we can compute explicitly the second fundamental form, showing that they are not asymptotic.  

In Section 7 we treat the case where the rank of $\zeta$ is equal to the Clifford index of the curve and  we give sufficient conditions ensuring that $\zeta$ is not asymptotic (see Theorem \ref{cliff=d}). 

In Section 8 we determine all asymptotic directions  of rank 1 and we show Theorem \ref{theorem 1} (Theorem \ref{rank1} and Theorem \ref{quintic}). 

In Section 9 we concentrate on deformations of rank 2. We show Theorem  \ref{theorem 2} (Theorem \ref{rank2}, Theorem \ref{sextic1}) and we give the example of asymptotic directions on bielliptic curves proving Theorem   \ref{theorem 3} (Theorem \ref{bielliptic}). 

In section \ref{Maroni} we consider trigonal curves of genus $g=6,7$ and Maroni degree 1 and trigonal curves of genus $g=5$,  showing that there can exist asymptotic directions that are not Schiffer variations in the ramification points of the $g^1_3$. We also  describe these asymptotic directions and we give the equation of the trigonal curves admitting such asymptotic  directions.

\section{Preliminaries on deformations}
\label{section2}

Recall that an infinitesimal  deformation $\zeta \in H^1(T_C)$ corresponds to a class of an extension 

\begin{equation}
\label{extension}
0 \to {\mathcal O}_C \to E \to K_C \to 0
\end{equation}

Taking gobal sections we have: 

$$ 0 \rightarrow H^0(C, {\mathcal O_C} )  \rightarrow H^0(C, E)  \rightarrow H^0(C, K_C) \stackrel{\cup \zeta} \rightarrow H^1(C, {\mathcal O_C}) \rightarrow ...$$

\begin{definition}
We define the rank of $\zeta$ as the rank of the map $\cup \zeta$.\end{definition}

We shall now describe in Dolbeault cohomology those deformations $\zeta$ having a given form $\omega\in H^0(K_C)$ in the kernel of $\cup \zeta$. 

Let $\omega\in H^0(K_C)$, $\omega\neq 0$ be a holomorphic 1-form, $Z=\sum n_ip_i$ its divisor, consider the sequence
 $$0\to T_C\stackrel{\omega}\to \cO_C\to \cO_Z\to 0,$$  and the corresponding exact sequence in cohomology: 
 \begin{equation}
 \label{dolbeault}
  0 \rightarrow H^0({\mathcal O}_C) \stackrel{i}\rightarrow H^0( {\mathcal O}_Z) \stackrel{\delta} \rightarrow H^1(T_C) \stackrel{\omega} \rightarrow H^1( {\mathcal O}_C) \rightarrow 0.
  \end{equation}
  
 By the above exact sequence, the elements in $ \zeta \in H^1(T_C)$ such that $ \omega \in ker(\cup \zeta)$ are exactly those belonging to the image of $\delta$. 
 
 We want now to give an explicit description in Dolbeault cohomology of the image of $\delta$. 
 
 Let $\{U_i, z_i\}$ be open pairwise disjoint coordinate neighbourhoods centred on $p_i $ and $p_i\in D_i\subset \Delta_i \subset U_i$ and $D_i$ and $\Delta_i$ are two closed disks
 where $D_i$ is in the interior of  $\Delta_i$.

  
Let $s\in H^0(\cO_Z)$ be a holomorphic section and assume that  $f_i(z_i) = \sum_{j=0}^{n_i-1} \beta_{j,i} z_i^{j}$ is its polynomial expression in $U_i$. Let $\tilde{\rho} $ be  a ${\mathcal C}^{\infty}$ function on $C$ which is equal to $1$ on $\cup_i D_i$ and equal to $0$ on $C\setminus \cup_i\Delta_i.$

Denote by  $\rho$ the ${\mathcal C}^{\infty}$ function on $C$ given by  $\sum_i \tilde{\rho}  f_i$.  
Let us introduce the following notation: writing  in a local coordinate $z$,  $\omega := g(z) dz$, we set $\frac{1}{\omega} = \frac{1}{g(z)} \frac{\partial}{\partial z}$. Hence $\frac{1}{\omega}$  defines a meromorphic section of $T_C$. 

Then,  in Dolbeault cohomology  we have $\delta(s) =[\debar(\frac \rho\omega ) ]$.  Setting $\zeta:=  [\debar(\frac \rho\omega ) ]$, clearly  $\omega \cup \zeta=0$. Conversely, by the exact sequence \eqref{dolbeault}, any $\zeta$ such that $\omega \cup \zeta=0$ is in the image of $\delta$, hence it has a Dolbeault representative as above.

 \begin{definition} 
 \label{split}
 We say that a nonzero  infinitesimal  deformation $\zeta \in H^1(T_C)$ is split if the vector bundle $E$ in \eqref{extension} splits as a direct sum of line bundles $E = (K_C \otimes L^{\vee}) \oplus  L$. 
  \end{definition}



We will now give a description in Dolbeault cohomology of split deformations. 

Take a form $\omega \in H^0(K_C)$ and  choose a decomposition of its zero divisor $Z = D +F$, where the supports of $D$ and $F$ are disjoint. 
 Set $[\sigma_D]\in H^0(\cO_Z)$,  where $\sigma_D\equiv 1$ on $\cup_{p_i \in Supp(D)} \Delta_i$ and  $\sigma_D \equiv 0$ on $\cup_{p_i \in Supp(F)} \Delta_i$. 


With the above notation set $\rho_D:=\tilde{\rho}  \sigma_D$. Setting $\Theta=\debar (\frac {\rho_D}{\omega}) $, we have in the exact sequence \eqref{dolbeault}, in Dolbeault cohomology $\delta(\sigma_D)=  \zeta =[\debar(\frac {\rho_D} {\omega}) ] \in H^1(C, T_C).$ One has $\Theta \omega = \debar \rho_D$, which is $\debar$-exact, hence $\omega \cup \zeta=0$ in $H^1(C, {\mathcal O}_C)$.

 \begin{proposition}
 \label{prop_split}
 An infinitesimal  deformation $\zeta \in H^1(T_C)$ is split if and only if $ \zeta = \delta(\sigma_D) =  [\debar(\frac {\rho_D} {\omega}) ] $ for some form $\omega = s \tau$, with zero divisor $Z = D + F$, where $D = div(s)$, $F = div(\tau)$, $L = {\mathcal O}_C(D)$ and $D$ and $F$ have disjoint supports. 
 In particular, if $D$ has degree $d$ and $h^0({\mathcal O}_C(D)) = r+1$, then $rank(\zeta) = d -2r$. 
 \end{proposition}
 
 \begin{proof}
Notice that if $\zeta$ is split, there exist $\tau \in H^0(K_C \otimes L^{\vee})$, $s \in H^0(L)$ non zero sections with disjoint zero loci such that the extension \eqref{extension} is given by: 

\begin{equation}
\label{extsplit}
0 \to {\mathcal O}_C \stackrel{(-\tau, s)} \longrightarrow (K_C\otimes L^{\vee} )\oplus L  \stackrel{s+\tau} \longrightarrow K_C \to 0
\end{equation}
hence $\omega: = s \tau \in ker (\cup \zeta)$. More precisely $ker (\cup \zeta) =  s \cdot H^0(K_C \otimes L^{\vee}) + \tau \cdot H^0(L). $ So, if $h^0(L) = r+1$ and $deg(L) = d$, $dim(ker (\cup \zeta) ) = h^0(L) + h^0(K_C \otimes L^{\vee}) - 1 = 2r + g-d$, hence $\zeta$ has rank equal to $d-2r$.



Setting $D = div(s)$, $F = div(\tau)$, we have $Z = D +F$, where the supports of $D$ and $F$ are disjoint. Setting  $[\sigma_D]\in H^0(\cO_Z)$ as above
then  $\delta(\sigma_D)=  \zeta' =[\debar(\frac {\rho_D} {\omega}) ] \in H^1(C, T_C).$ One has $\debar(\frac {\rho_D} {\omega}) \omega = \debar \rho_D$, which is $\debar$-exact, hence $\zeta'\omega=0$ in $H^1(C, {\mathcal O}_C)$. 
 
 We claim that $\zeta' = \zeta$. In fact, consider the extension \eqref{extsplit} given by $\zeta$ and tensor it by $T_C$. 
 We get 
 
 \begin{equation}
 0 \to T_C \stackrel{(-\tau, s)} \longrightarrow L^{\vee} \oplus (T_C \otimes L  )\stackrel{s+\tau} \longrightarrow {\mathcal O}_C \to 0.
 \end{equation}
 Then $\zeta = \delta(1)$, where $\delta: H^0({\mathcal O}_C) \to H^1(T_C)$ is the coboundary morphism. 
 We have $s \frac{(1-\rho_D)}{s} + \tau \frac{\rho_D}{\tau} = 1$, hence the element $Y:= (\frac{1-\rho_D}{s}, \frac{\rho_D}{\tau}) \in {\mathcal C}^{\infty}(L^{\vee} \oplus (T_C \otimes L  ))$ is mapped to $1$ under the map $(s+\tau)$.  Then $(s + \tau)\debar (Y) = \debar((s+\tau)(Y) )= \debar(1) =0 \in  {\mathcal A}^{0,1}_C$, so there exists an element $X \in {\mathcal A}^{0,1}_C(T_C)$ such that $(-\tau X, sX) = \debar(Y) = (\debar(\frac{1-\rho_D}{s}), \debar(\frac{\rho_D}{\tau}))   \in {\mathcal A}^{0,1}_C(L^{\vee}) \oplus {\mathcal A}^{0,1}_C(T_C \otimes L))$ and  $\zeta = [X] \in H^1(T_C)$. Thus if we take $X = \debar(\frac{\rho_D}{\omega})$, we get $$(- \tau \debar(\frac{\rho_D}{\omega}), s \debar(\frac{\rho_D}{\omega})) = (- \debar(\frac{\rho_D}{s}), \debar(\frac{\rho_D}{\tau})) = (\debar(\frac{1-\rho_D}{s}), \debar(\frac{\rho_D}{\tau})),$$
 since $\debar(\frac{1}{s}) =0$, as $s$ is holomorphic. So we have shown that $\zeta = [\debar (\frac{\rho_D}{\omega})] = \zeta'$. 
 
 Viceversa, assume that in the exact sequence \eqref{dolbeault}  the divisor $Z$ of $\omega$ has a splitting  $Z = D + F$, where $D$ and $F$ have disjoint support. Then $\omega = s \tau$ where $s \in H^0({\mathcal O}_C(D))$, $\tau \in H^0({\mathcal O}_C(F))$, with $div(s) =D$ and $div(\tau) = F$. 
Define $\rho_D$ as above and $\zeta:= [\debar (\frac{\rho_D}{\omega})] $, then clearly $\omega \in ker(\cup \zeta)$. 

We claim that the deformation $\zeta$ is split. In fact taking $L = {\mathcal O}_C(D)$, the extension \eqref{extsplit} has a class given by $\zeta$.

\end{proof}

 Let $C$ be a smooth complex projective curve of genus $g \geq 2$. Let $p$ be  a point in $ C$ and $z$ a local coordinate  centred in $p$. We define the $n^{th}$ Schiffer variation  at $p$ to be the element $\xi_p^n \in H^1(C, T_C) \cong H^{0,1}_{\bar{\partial}}(T_C)$ whose Dolbeault representative is $\frac{ \bar{\partial}\rho_p}{z^n} \frac{\partial}{\partial z}$, where $\rho_p$ is a bump function in $p$ which is equal to one in a small neighborhood $U$ containing $p$, $\xi_p^n = [\frac{ \bar{\partial}\rho_p}{z^n} \frac{\partial}{\partial z}]$. Clearly  $\xi_p^n$ depends on the choice of the local coordinate $z$. Take the exact sequence 
	$$0 \rightarrow T_C  \rightarrow T_C(np)  \rightarrow T_C(np)_{|np} \rightarrow 0, $$
and the induced exact sequence in cohomology: 
$$ 0 \rightarrow H^0(T_C(np)) \rightarrow H^0(T_C(np)_{|np}) \stackrel{\delta^n_p}\rightarrow H^1(T_C). $$
By Riemann Roch, if $n <2g-2$ we have: $$ h^0(T_C(np)) =0,$$ hence we have an inclusion $$ \delta^n_p: H^0(T_C(np)_{|np})  \cong {\mathbb C}^n \hookrightarrow H^1(T_C)$$ and the image of $\delta^n_p$ in $H^1(C,T_C)$ is the $n$-dimensional subspace $\langle \xi_p^1,...,\xi_p^n \rangle$.	

Clearly $H^0(K_C(-np)) \subset Ker (\cup \xi_p^n)$, for $n \leq g$, hence $\xi_p^n$ has rank $\leq n$. For $n=1$ the first Schiffer variations  $\xi_p^1$ are the usual Schiffer variations, that we denote simply by $\xi_p$, and they have rank 1 (see Section 7).

More generally, any linear combination $\zeta =  \sum_{i=1}^k \sum_{j = 1}^{m_i} b_{j,i} \xi^{j}_{p_i} $, with $ \sum_{i=1}^k m_i \leq g$ has rank $\leq \sum_{i=1}^k  m_i$, since $H^0(K_C(-  \sum_{i=1}^k m_i p_i)) \subset ker(\cup \zeta)$. 




\section{Some computations}
 We will now show some technical results that will be useful in the sequel to make explicit computations on the second fundamental form of the Torelli map.  

Consider two holomorphic one forms $\omega_1$ and $\omega_2$.  With the notation introduced in section \ref{section2}, take an infinitesimal deformation $\zeta$ having a form $\omega \neq 0$ in its kernel and denote by $Z = \sum n_i p_i$ its zero divisor . Then $\zeta = [\Theta] \in H^1(T_C)$ with $ \Theta:= \debar(\frac \rho\omega )$, and $\rho = \sum_i \tilde{\rho}  f_i$.  
We have $ \Theta\omega_i= \gamma_i + \debar h_i$, where $\gamma_i$ is harmonic and  $h_i$ is a ${\mathcal C}^\infty$ function on $C$, so  $\zeta \cup \omega_i=[\gamma_i] \in H^1({\mathcal O}_C)$. 
Consider the two meromorphic functions  $g_1=\frac{\omega_1}\omega$ and $g_2=\frac{\omega_2}\omega$. 
Then clearly, by the above construction we have 

$$ \Theta\omega_i=  \gamma_i + \debar h_i = \debar (\rho g_i).$$
 Note  that if $ \rho g_1$ is $C^\infty$, then  $\omega_1 \in ker ( \cup \zeta)$ and $\gamma_1 =0$. Moreover,  up  to a constant,  we have $h_1=  \rho g_1.$

 Set
\begin {equation}\label{eterm}
w(\zeta,\omega_1,\omega_2 ):=2\pi i \sum_{p_i \in Supp(Z)}Res_{p_i} (f_ig_1d(f_ig_2)) .
\end{equation}

In the next section we will show that this expression will be very useful in the computation of the second fundamental form of the Torelli map.  
  \begin{remark}
  \label{h1h2}
 For any ${\mathcal C}^\infty$ functions $h_1, h_2$, we have 
 
 $$\int_C \partial h_1 \wedge \overline{\partial}h_2 = \int_C \partial h_2 \wedge \overline{\partial}h_1.$$ 
 \end{remark}
 \begin{proof}
 By Stokes Theorem we have: 
 $$\int_C \partial h_1 \wedge \overline{\partial}h_2 =\int_C d(h_1 \overline{\partial}h_2) - \int_C h_1 \partial \overline{\partial}h_2 = 
   \int_C h_1 \overline{\partial} \partial h_2=$$
 $$  = \int_C d(h_1  \partial h_2) - \int_C \overline{\partial}h_1 \wedge \partial h_2 = \int_C \partial h_2 \wedge \overline{\partial} h_1. $$
 \end{proof}
 \begin{lemma}  Assume  that $ \rho g_1$ is $C^\infty$, then we have have\label{formula}
 $$\int_C \partial h_1\wedge \debar h_2=\int_C  \partial( \rho g_1)\wedge \debar (\rho g_2)=$$
 $$ =\int_C  \partial(\rho g_2)\wedge\debar (\rho g_1) +w(\zeta,\omega_1,\omega_2 ) =  \int_C \partial h_2 \wedge \overline{\partial}h_1.$$
 
 In particular, if both $ \rho g_1$ and $ \rho g_2$ are $C^\infty$, we have $w(\zeta,\omega_1,\omega_2 )=0$. 
\end{lemma}

\begin{proof}

Since $\int_C \partial h_1\wedge \gamma_2=\int_C d(h_1 \gamma_2)=0$,  and $\Theta \omega_2 = \bar{\partial} (\rho g_2)$, we  get
  $$\int_C \partial h_1\wedge \debar h_2=\int_C  \partial h_1\wedge(\Theta \omega_2-\gamma_2)=\int_C  \partial h_1\wedge\Theta\omega_2 =$$
$$=\int_C  \partial (\rho g_1)\wedge \debar (\rho g_2),$$
hence the first equality is proven.  Moreover, 

$$\int_C  \partial (\rho g_1)\wedge \debar (\rho g_2)= \int_C d(\rho g_1)\wedge\debar(\rho g_2)=\int_Cd(\rho g_1\debar(\rho g_2))-\int_C\rho g_1\partial(\debar (\rho g_2 )).$$

Using Stokes theorem, and recalling that $\debar(\rho g_2) = \Theta\omega_2$ is a ${\mathcal C}^{\infty}$ $(0,1)$-form on $C$, we have 
$$\int_Cd(\rho g_1\debar (\rho g_2))=0. $$

Then we get

$$\int_C \partial h_1\wedge \debar h_2=-\int_C \rho g_1\partial\debar (\rho g_2)=\int_C \rho g_1\debar \partial(\rho g_2)=\int_Cd(\rho g_1  \partial(\rho g_2))- \int_C\debar (\rho g_1)\wedge \partial(\rho g_2)= $$
$$=\int_Cd(\rho g_1  \partial(\rho g_2))+ \int_C \partial(\rho g_2) \wedge  \debar (\rho g_1).$$

Now, taking $D_{i, \epsilon}$   a small disc around $p_i$ where $\rho = f_i$, and by Stokes theorem, we have $$\int_Cd(\rho g_1  \partial(\rho g_2))=\sum \lim_{\epsilon\to 0} \int_{\partial D_{i,\epsilon}}  \rho g_1d( \rho g_2)=
2\pi i \sum_{p_i \in Z}Res_{p_i} (f_ig_1d(f_ig_2))= w(\zeta,\omega_1,\omega_2 ). $$

So we have shown that 
$$\int_C \partial h_1\wedge \debar h_2
  =\int_C  \partial(\rho g_2)\wedge\debar (\rho g_1) +w(\zeta,\omega_1,\omega_2 ),$$
and by   Remark \ref{h1h2} the first term is equal to 
  $ \int_C \partial h_2 \wedge \overline{\partial}h_1.$

 Finally,  if $\rho g_1$ and $\rho g_2$ are both $C^\infty$, clearly $f_ig_1d(f_ig_2)$ has no poles on $Z$, hence the last sentence follows.

 \end{proof}

\begin{remark}
 Notice that if $\zeta= [\debar (\frac{\rho_D}{\omega})] $ is split,  in formula \eqref{eterm} we have $$w(\zeta,\omega_1,\omega_2 )=2\pi i \sum_{p_i \in Supp(D)}Res_{p_i} (g_1d(g_2)) .$$
 \end{remark}

\section{Second fundamental form and Hodge Gaussian map}

 Denote by ${\mathcal M}_g$  the moduli space of smooth complex projective curves of genus $g$ and  let ${\mathcal A}_g$ be the moduli space of principally polarised abelian varieties of dimension $g$. 
 The space ${\mathcal A}_g$ is a quotient of the Siegel space ${\mathcal H}_g = Sp(2g, {\mathbb R})/U(g)$ by the action of the symplectic group $Sp(2g, {\mathbb Z})$. The space ${\mathcal H}_g$ is a Hermitain symmetric domain and it is endowed with a canonical symmetric metric. We consider the induced  orbifold metric (called the Siegel metric) on the quotient ${\mathcal A}_g$. 
 
Denote by 
$$j: {\mathcal M}_g \to {\mathcal A}_g,  \ [C] \mapsto [j(C), \Theta_C]$$
the Torelli map, where $j(C)$ is the Jacobian of $C$ and $\Theta_C$ is the principal polarisation induced by cup product. If $g \geq 3$, it is an orbifold embedding outside the hyperelliptic locus (\cite{oort-st}).  

Consider the complement  ${\mathcal M}_g^0$ of the hyperelliptic locus in ${\mathcal M}_g$ and the cotangent exact sequence of the Torelli map: 
$$ 0 \to N_{{\mathcal M}_g^0/{\mathcal A}_g} ^*\to \Omega^1_{{\mathcal A}_g{|{\mathcal M}_g^0}} \stackrel{q}\to \Omega^1_{{\mathcal M}_g^0} \to 0,$$
where $q = dj^*$ is the dual of the differential of the Torelli map. 
Call  $\nabla$  the Chern connection on $\Omega^1_{{\mathcal A}_g{|{\mathcal M}_g^0}}$ with respect to the Siegel metric and let 
$$II:  N_{{\mathcal M}_g^0/{\mathcal A}_g} ^* \to Sym^2 \Omega^1_{{\mathcal M}_g^0}, \ \ II = (q \otimes Id_{\Omega^1_{{\mathcal M}_g^0}}) \circ \nabla_{|N_{{\mathcal M}_g^0/{\mathcal A}_g} ^*}$$
be the second fundamental form of the above exact sequence. 

From now on we will assume $g \geq 4$, so $N_{{\mathcal M}_g^0/{\mathcal A}_g} ^*$ is non trivial.
  
 If we take a point  $[C] \in {\mathcal M}_g^0$, we have: 
 $$ N_{{\mathcal M}_g^0/{\mathcal A}_g, [C]} ^*= I_2, 
 \ \Omega^1_{{\mathcal A}_g{|{\mathcal M}_g^0}, [C]} = Sym^2H^0(C, K_C), \ \Omega^1_{{\mathcal M}_g^0, [C]} = H^0(C,K_C^{\otimes 2}),$$
 where $I_2:= I_2(K_C)$ is the vector space of quadrics containing the canonical curve and the dual of the differential of the Torelli map $q$ is the multiplication map of global sections. Then, at the point $[C]$, the second fundamental form is a linear map
 $$II: I_2 \to Sym^2H^0(K_C^{\otimes 2}). $$

\begin{definition}
A nonzero  element $\zeta \in H^1(T_C)$ is an asymptotic direction if 
$$II(Q)(\zeta \odot \zeta) = 0$$
 for every $Q \in I_2$. 
\end{definition}
 In  \cite[Theorem 2.1]{cpt} it is proven that $II$ is equal (up to a constant) to the Hodge-Gaussian map $\rho$ of  \cite[Proposition-Definition 1.3]{cpt}. We briefly recall its definition.


 Let $Q= \sum_{i, j}a_{i,j}\omega_i\otimes \omega_j\in I_2$, with $a_{i,j} = a_{j,i} $, $\omega_i \in H^0(K_C)$, 
and $\Theta\in \cA^{0,1}(T_C)$,  $[\Theta]=:\zeta.$

We write 
$\Theta \omega_i=\gamma_i +\debar h_i$ where $\gamma_i$ is a harmonic $(0,1)$-form.
Now, identifying  $Sym^2H^0(K_C^{\otimes 2})$ with the symmetric homomorphisms $H^1(T_C) \to H^0(K_C^{\otimes 2})$, we have  (see \cite{cpt}):
\begin{equation}
\label{rho}
II(\sum_{i, j}a_{ij}\omega_i \otimes \omega_j)(\zeta)=\sum_{i,j}a_{ij}\omega_i\partial h_{j}.
\end{equation}

In \cite[Theorem 3.1]{cpt} (see also \cite[Theorem 2.2]{cfg}) it is proven that if 
 $C$ is a non-hyperelliptic
curve of genus $g \geq 4$ and $p,q$ are two distinct points in $C$,  we have
 \begin{equation}
 \label{xip}
 II(Q)(\xi_p \odot\xi_p)  = -2\pi i  \cdot \mu_2(Q)(p),
 \end{equation}
 where $\mu_2: I_2 \rightarrow H^0(K_C^{\otimes 4}) $ is the second Gaussian map of the canonical bundle 
 (see \cite{cpt}, or \cite{cfg} for more details). \\

\begin{remark}
Since $II$ is injective (see  \cite[Corollary 3.4]{cf-trans}), $II(I_2) \subset Sym^2H^1(T_C)^{\vee}$ is a linear system of quadrics in ${\mathbb P}H^1(T_C) \cong {\mathbb P}^{3g-4} $ of dimension $\frac{(g-2)(g-3)}{2}$.  Hence for every curve $C$ of genus $g \leq 9$, $dim( II(I_2)) < 3g-4$, so the intersections of the quadrics in $II(I_2)$ is non empty, thus there exist asymptotic directions. 
\end{remark}

{\bf Examples of asymptotic directions.}
\begin{enumerate}
\item Using equation \eqref{xip} one can see  that Schiffer variations $\xi_p$ at ramification points 
of a $g^1_3$ on any trigonal curve of genus $g \geq 4$ are asymptotic directions (see Lemma \ref{lem_schiffer}). Moving a branch point in ${\mathbb P}^1$, one can see that there exist algebraic curves in the trigonal locus having these Schiffer variations as tangent directions (see \cite{gt}). 

\item Other examples of asymptotic directions of rank 1  different from Shiffer variations at ramification points of a $g^1_3$ on  trigonal curves of genus 5, or of genus 6,7 with Maroni degree $k=1$  are given in the last sections.  We will  give an explicit  description of these loci of trigonal curves admitting such asymptotic directions. 

\item In Theorem \ref{bielliptic} we prove that on any bielliptic curve of genus at least 5 there exist linear combinations of two Schiffer variations that are asymptotic directions of rank 2. 

\item Other examples of asymptotic directions are given by tangent vectors to special (hence totally geodesic) subvarieties of ${\mathcal A}_g$ generically contained in the Torelli locus  (see \cite{dejong-noot}, \cite{fgp}, \cite{fgs}, \cite{fpp}, \cite{moonen-special}, \cite{moonen-oort}, \cite{rohde}, \cite{shimura}, \cite{spelta}). In all these examples, $g \leq 7$. 

\end{enumerate}

\begin{remark} 
\label{xiao}
We recall that the original Xiao conjecture for non isotrivial fibered surfaces $S\to B$ with general fibre $F$ of genus $g$
says  that the relative irregularity $q_f$ safisfies the inequality $q_f\leq \frac {g+1}{2}$.  In \cite{pirola}, \cite{albano-pirola} counterexamples to this conjecture have been given, and a modified version, 
$q_f \leq \left \lceil{\frac {g+1}{2}}\right \rceil$ is stated in \cite{bnga}.  In two of the four counterexamples, the associated modular map $p:B\to M_g$  gives rise to curves contained in a totally geodesic subvariety (\cite{fpp}, \cite{grumoe}, \cite{spelta}).  Moreover this always happens when
 $q_f=g-1.$  In fact in that case there is a fixed part $A\times B \to B$ of the Jacobian fibration $J\to B$ where $A$ is an abelian variety of dimension $g-1.$ Then the  period variation of  the Jacobian fibration is only given by the elliptic fibration given by the quotient $J/A.$  So the image of the period map $B  \to A_g$  is a totally geodesic curve.
Since the tangent vectors  to a totally geodesic curve are asymptotic directions, 
these fibrations  give rise to  asymptotic directions whose rank is  $\leq g-q_f.$
So the vanishing of the second fundamental form on tangent directions give new infinitesimal constraints that could help to classify fibered surfaces with $q_f > \frac{g+1}{2}$.  
\end{remark}

Now consider a quadric $Q \in I_2$. Notice that we can always assume that $Q$ has the following expression: 
 $$Q = \sum_{i=1}^s\omega_{2i-1} \odot \omega_{2i}.$$ 
 Fix a holomorphic form $\omega\in H^0(K_C),$ $\omega\neq 0,$ with zero locus $Z$. With the notation of section \ref{section2}, we set  $\zeta =[\debar(\frac {\rho}{ \omega})] $. So we have 
$\zeta  \omega=0$ and $\Theta \omega_j= \gamma_j+\debar h_j$, $\gamma_j$ harmonic.  Let $g_i$ be the meromorphic function given by $\frac{\omega_i}{\omega}$. 
We would like to compute $II(Q)(\zeta \odot \zeta).$  
\begin{proposition}
\label{w}
Assume
$\rho g_{2i-1}$ is $ C^{\infty},$ $\forall i=1,...,s$.  Using the notation of (\ref{eterm}) we have: \label{compiti}
$$II(Q) (\zeta \odot \zeta)= -\sum_{i=1}^s w(\zeta, \omega_{2i-1},\omega_{2i}).$$
\end{proposition}\begin{proof}
From the Hodge Gaussian computation \eqref{rho}, we get 
$$II(Q)(\zeta)=\sum_{i=1}^s (\omega_{2i-1}\partial h_{2i}+\omega_{2i}\partial h_{2i-1}).$$ 
We have 
 $$II(Q)(\zeta \odot \zeta)= \zeta(II(Q)(\zeta))=  \debar(\frac {\rho}{ \omega})(\sum_{i=1}^s (\omega_{2i-1}\partial h_{2i}+\omega_{2i}\partial h_{2i-1}))=$$
$$ = \sum_{i=1}^s\int_C((\gamma_{2i-1} + \debar h_{2i-1})\wedge\partial h_{2i}+ (\gamma_{2i} + \debar h_{2i})\wedge\partial h_{2i-1})=$$
$$\sum_{i=1}^s\int_C(\debar h_{2i-1}\wedge\partial h_{2i}+  \debar h_{2i}\wedge\partial h_{2i-1}) = - \sum_{i=1}^s\int_C(\partial h_{2i} \wedge \debar h_{2i-1}+ \partial h_{2i-1} \wedge \debar h_{2i}),$$
since $\gamma_k \wedge \partial h_l$ is exact for any $k, l$.

By Remark \ref{h1h2} we have

$$II(Q)(\zeta \odot \zeta)= - \sum_{i=1}^s\int_C(\partial h_{2i} \wedge \debar h_{2i-1}+ \partial h_{2i-1} \wedge \debar h_{2i})= -2 \sum_{i=1}^s\int_C(\partial h_{2i-1} \wedge \debar h_{2i}).$$

Using Lemma \ref{formula} we have 
$$\int_C(\partial h_{2i-1} \wedge \debar h_{2i})=  \int_C\partial(\rho g_{2i-1})\wedge\debar (\rho g_{2i})= w(\zeta,\omega_{2i-1},\omega_{2i}) +\int_C \partial(\rho g_{2i})\wedge\debar (\rho g_{2i-1}).$$

So we obtain 
$$II(Q)(\zeta \odot \zeta)=
-\sum_{i=1}^s [(w(\zeta,\omega_{2i-1},\omega_{2i}) +\int_C \partial(\rho g_{2i})\wedge\debar (\rho g_{2i-1}))+ \int_C\partial(\rho g_{2i-1})\wedge\debar (\rho g_{2i})]. $$

Now the result follows since
$$\Omega:=\sum (\partial(\rho g_{2i})\wedge\debar (\rho g_{2i-1})+\partial(\rho g_{2i-1})\wedge\debar (\rho g_{2i})) =0.$$
In fact we have $\debar (\rho g_j)=g_j\debar\rho$  so that
$$\Omega= \sum (\partial(\rho g_{2i})g_{2i-1} + \partial(\rho g_{2i-1})g_{2i}) \wedge \debar \rho= \sum  (g_{2i}g_{2i-1} + g_{2i-1}g_{2i})\partial  \rho  \wedge \debar \rho+$$
$$+\sum  ((\partial g_{2i}) g_{2i-1} + (\partial g_{2i-1})g_{2i}) \wedge  \rho  \debar \rho=0,$$
since $Q \in I_2$, hence   $\sum (g_{2i-1} g_{2i} + g_{2i-1}g_{2i}) = 0,$ and hence also its derivative $\sum  ((\partial g_{2i}) g_{2i-1} + (\partial g_{2i-1})g_{2i}) =0$.

\end{proof}

\section{Deformations of rank $d < Cliff(C)$}

Recall that if $C$  is a smooth projective curve of genus $g$ and $L$ is a line bundle on $C$ 
then  the Clifford index of $L$ is 
$$Cliff(L) = deg(L) -2h^0(L) + 2, $$
and the Clifford index of $C$ is
$$Cliff(C) = min_{ L \in Pic(C)}\{deg(L) - 2h^0(L) + 2 :  \ h^0(L) \geq 2, h^1(L) \geq 2 \}.$$
We say that a line bundle $L$ contributes to the Clifford index if $h^0(L) \geq 2, h^1( L) \geq 2$. 

One always has $Cliff(C) \geq 0$, and $Cliff(C) = 0$ if and only if $C$ is hyperelliptic; $Cliff(C) = 1$ if and only if $C$ is trigonal or isomorphic to a plane quintic, and $Cliff(C) = 2$ if and only if $C$ is tetragonal,
or  isomorphic to a plane sextic (see \cite{martens}). We have the following relation between the Clifford index and the gonality $gon(C)$ of a curve $C$ (\cite{coppens-martens}): 
$$Cliff(C) + 2 \leq gon(C) \leq  Cliff(C) + 3.$$

Consider an infinitesimal deformation $\zeta \in H^1(T_C)$ and a corresponding extension 

\begin{equation}
\label{extensionD}
0 \to {\mathcal O}_C \to E \to K_C \to 0.
\end{equation}

\begin{theorem}
\label{cliff-rank}
Let $C$ be a smooth algebraic curve of genus $g \geq 4$ and $\zeta \in H^1(T_C)$ a deformation of rank $d>0$.
Suppose one of the following assumptions is  satisfied: 
\begin{enumerate}
\item $d < Cliff(C)$, 
\item $d = Cliff(C) < \frac{g-1}{2}$ and $E$  not globally generated. 
\end{enumerate}

Then $\zeta$ is a linear combination of Schiffer variations supported on an effective divisor $D$ of degree $d$. 

\end{theorem}

\begin{proof}

Taking gobal sections in the extension \eqref{extensionD} corresponding to the rank $d$ deformation $\zeta$, we have: 

$$ 0 \rightarrow H^0(C, {\mathcal O_C} )  \rightarrow H^0(C, E)  \rightarrow H^0(C, K_C) \stackrel{\cup \zeta} \rightarrow H^1(C, {\mathcal O_C}) \rightarrow ...$$

Since $\zeta$ has rank $d$, we get $h^0(E) = g-d +1$. By a Theorem of Segre-Nagata  and Ghione (see \cite{LazII} p. 84) there exists a subline bundle $A$ of $E$ such that $deg(A) \geq \frac{g-1}{2}$. So, up to saturation we have a diagram 

$$
\xymatrix{
& & 0\ar[d] &\\
             &  & A \ar[d]^{\iota} \ar[dr]^f& & \\
    0 \ar[r]& {\mathcal O}_C \ar[r] \ar[dr]^h& E\ar[d]\ar[r]& K_C\ar[r] & 0 \\
     & & K_C \otimes A^{\vee}\ar[d]& \\
     &  & 0& \\
    &&&& }           
$$
Both the maps $f$ and $h$ are nonzero, since $deg(A) >0$. In fact, if $f$ were zero,  then $\iota$ would factor through ${\mathcal O}_C$. So $f$ is nonzero and hence also $h$ is nonzero.

Hence we have  $h^0(K_C \otimes A^{\vee}) \geq 1$. We claim that $h^0(A) \geq 2$. In fact, if $h^0(A) \leq 1$, we would have $h^0(K_C \otimes A^{\vee})  \geq h^0(E) - 1 = g-d$ and by Riemann Roch, we would get $deg(A) \leq d  < \frac{g-1}{2}$, a contradiction. 
So $h^0(A) \geq 2$, and if $h^0(K_C \otimes A^{\vee}) \geq 2$, $A$ contributes to the Clifford index and we have: 
$$h^0(A) - h^0(K_C \otimes A^{\vee}) = deg(A) -g+1,$$
$$h^0(A) + h^0(K_C \otimes A^{\vee}) \geq g-d+1.$$
So summing up we get 
$$2 h^0(A) \geq deg(A) -d +2,$$
hence $Cliff(A) \leq d$. If $d < Cliff(C)$, this is a contradiction. If  $d =Cliff(C)$,  then $Cliff(A) =d$,  $h^0(A) + h^0(K_C \otimes A^{\vee}) =g-d+1,$ 
and we have an exact sequence 
$$0 \to H^0(A)  \to H^0(E) \to H^0(K_C \otimes A^{\vee}) \to 0.$$
Moreover $A$ and $K_C \otimes A^{\vee}$ are both base point free (since they compute the Clifford index). Thus $E$ is globally generated, a contradiction. 

So $h^0(K_C \otimes A^{\vee}) =1$, hence  there exists an effective divisor $D$ such that $K_C \otimes A^{\vee} = {\mathcal O}_C(D)$, and $h^0({\mathcal O}_C(D)) =1$. 
So  we have $h^0(A) = h^0(K_C(-D)) =  g-d$, hence by Riemann Roch, $deg(A) = deg(K_C(-D))  = 2g-2-d$. Thus the image of the map $H^0(E) \to H^0(K_C)$ is $H^0(K_C(-D))$ and we get a commutative diagram of extensions
\begin{equation}\label{hodge1}
\begin{tikzcd}
0\arrow{r} & {\mathcal O} \arrow{r}  \arrow{d}{=}& E'
\arrow{r}  \arrow[hook]{d}&K_C (-D) \arrow{r} \arrow{dl}{\iota} \arrow[hook]{d}{f}&0 \\
0\arrow{r}& {\mathcal O}  \arrow{r}& E
\arrow{r} & K_C\arrow{r} & 0\\
\end{tikzcd}
\end{equation}
where the upper extension splits, since the image of $\iota$ is contained in $E'$ by construction. Hence it corresponds to $0 \in H^1(T_C(D))$ and 
 the element  $\zeta \in H^1(T_C)$ belongs to the kernel of the map $H^1(T_C) \to H^1(T_C(D))$ which is the image of the injective map $H^0(T_C(D)_{|D}) \to H^1(T_C)$, that is  the space of Schiffer variations supported on $D$. 

\end{proof}

\begin{remark}
Notice that Theorem \ref{cliff-rank} is a generalisation of the generic Torelli theorem of Griffiths. In fact, denote by $C_d$ the symmetric product of $C$. Then, under the assumption $Cliff(C) > d$, we characterise the image of the natural map 
$${\mathbb P}T_{C_d}\to {\mathbb P}(H^1(T_C))$$
as the locus of deformations of rank $\leq d$. When $d=1$ it is the bicanonical curve. 
\end{remark}
We will now consider linear combination of Schiffer variations supported on a divisor of degree less than the Clifford index of $C$. First recall the following 
\begin{lemma}
\label{castelnuovo}
Let $C$ be a smooth curve of genus $g$,  take a positive integer $d < Cliff(C)$  and an effective divisor $D$  of degree $d$. Then $K_C(-D)$ is projectively normal, so the multiplication map 
\begin{equation}
\label{mult}
m:Sym^2H^0(C, K_C(-D)) \to H^0(K_C^{\otimes 2}(-2D))
\end{equation}
is surjective. 
\end{lemma}
\begin{proof}
Notice that  $h^0(D) =1$  since $Cliff(C)>d$. Moreover 
$K_C(-D)$ is very ample by Riemann Roch and by the assumption $Cliff(C)>d$. 
So by  \cite[Theorem 1]{GL1}, $K_C(-D)$ is projectively normal. 
\end{proof}

 We have  the following 
\begin{theorem}
\label{no-schiffer-cliff}
Let $C$ be a smooth curve of genus $g$,  take a positive  integer $d < Cliff(C)$, an effective divisor $D = \sum_{i=1}^k m_i p_i$  of degree $d$. Then any linear combination of all Schiffer variations $\xi^{n_i}_{p_i} $, with $n_i \leq m_i$ is not an asymptotic direction. 

\end{theorem}

\begin{proof}
We will prove the statement by induction on $d$. The case $d =1$ is already known. In fact,  in \cite{cpt}  (see also \cite[Theorem 2.2]{cfg})  it is proven that for any $Q \in I_2(K_C)$, we have $$II(Q)(\xi_p \odot \xi_p) = -2 \pi i \mu_2(Q)(p),$$
and by \cite[Theorem 6.1]{cf-Michigan} we know that if $C$ is not hyperelliptic, and not trigonal of genus $g \geq 5$, for any $p \in C$, there exists a quadric $Q \in I_2(K_C)$ such that $\mu_2(Q)(p) \neq 0$. 

By induction assume that no linear combinations of all $\{\xi^{l_i}_{p_i} \ | \  l_i \leq m_i,  \ i <k,   \ l_k \leq m_{k}-1  \}_{i =1,...,k}$ are asymptotic. 

So take a linear combination $\zeta:= \sum_{i=1}^k \sum_{j = 1}^{m_i} b_{j,i} \xi^{j}_{p_i} $. 
To prove the result, we will show that there exists a quadric $Q \in I_2$ such that $II(Q) (\zeta \odot \zeta) = 0$  if and only if $b_{m_k, k} =0$. Then if $ \zeta$ were asymptotic, $b_{m_k,k} =0$, and by induction we get a contradiction. 

Notice that  $h^0(D) =1$  since $Cliff(C)>d$. Set $F_d := D$, $F_{d-1} := F_d - p_k$. 
By Lemma \ref{castelnuovo} the following multiplication maps  $$\mu_{F_{d-1}}: Sym^2 H^0(K_C(-F_{d-1}))\to H^0(K_C^{\otimes 2}(-2F_{d-1})),$$
$$\mu_{F_d}:Sym^2 H^0(K_C(-F_d))\to H^0(K_C^{\otimes 2}(-2F_d))$$
are surjective.

Set $$I_{d-1} =\ker{\mu_{F_{d-1}}}$$ and $$I_{d}=\ker{\mu_{F_d}}.$$
We have and inclusion $I_{d}\subset I_{d-1}$ and $$\dim (I_{d-1}/I_{d})=g-d-1.$$ 
We would like to fix a basis of $H^0(K_C(-F_{d-1})).$ Recall that by Riemann Roch and the assumption $Cliff(C) >d$, both $K_C(-F_d)$ and $K_C(-F_{d-1})$ are very ample. 
We first take a form $\omega_1\in H^0(K_C(-F_{d-1})) \setminus H^0(K_C(-F_{d}))$, $\omega_2 \in H^0(K_C(-F_d))$ such that $ord_{p_i} \omega_2 =m_i,$ $\forall i =1,...,k$, and $\omega_3 \in H^0(K_C(-F_d))$ such that $ord_{p_k} \omega_3 = m_k +1$. 
Finally   $\omega_4, \dots, \omega_{g-d+1} \in H^0(K_C(-F_{d}))$ such that 
 $ord_{p_k} \omega_i >m_k +1 $, $\forall i =4,..., g-d+1$.

Then any $Q \in I_{d-1}$ can be written as follows: 
$$Q = \sum_{j =1}^{g-d +1} \alpha_j \omega_1 \odot \omega_j + \cK,$$
where  $\cK\in Sym^2(H^0(K_C(-F_{d}))).$

Notice that $\mu_{F_{d-1}}(Q) =0,$  $\alpha_1 \omega_1^2$ is the only term whose vanishing order in $p_k$ is $2m_k-2$, and $\alpha_2 \omega_1 \omega_2$ is the only term whose vanishing order in $p_k$ is $2m_k-1$. So we must have $\alpha_1 = \alpha_2 =0$. Since $\dim (I_{d-1}/I_{d})=g-d-1,$ there exists a quadric $Q \in I_{d-1}$ such that $\alpha_3 \neq 0$. But then since  $ord_{p_k} (\omega_1 \omega_3) =2m_k = ord_{p_k} (\omega_2^2)$ and all the other terms have higher vanishing order in $p_k$, the quadric $Q$ has the form 
$$Q = \alpha_3 \omega_1 \odot \omega_3 + \beta \omega_2 \odot \omega_2 +   \cK',$$
where $ \cK' = \sum_{j \geq 4} \alpha_j \omega_1 \odot \omega_j + \cK''$, with $\cK'' \in Sym^2(H^0(K_C(-F_{d})))$ and $\alpha_3 \neq 0$, $\beta \neq 0$. 

Observe now that we have $\zeta =  \sum_{i=1}^k \sum_{j = 1}^{m_i} b_{j,i} \xi^{j}_{p_i}  = [ \debar(\sum_{i=1}^k \sum_{j = 1}^{m_i} b_{j,i} \frac{z_i^{m_i-j} \rho_{p_i}} {\omega_2})]$, where $\rho_{p_i}$ are  bump functions which are equal to 1 in a neighbourhood of $p_i$ and $z_i$ is a local coordinate around $p_i$. In the notation of section \ref{section2}, $\rho = \sum_{i=1}^k \sum_{j = 1}^{m_i} b_{j,i} z_i^{m_i-j} \rho_{p_i}$.

Using the formula in Proposition \ref{w} we get 
$$II(Q) (\zeta \odot \zeta)= \alpha_3 b_{m_k, k} Res_{p_k}  \frac { \omega_3}{\omega_2} d(\frac { \omega_1}{\omega_2}) = c  \cdot  b_{m_k, k} Res_{0} (z_k d(\frac{1}{z_k})) = -c b_{m_k, k},$$ 
where $c \neq 0$ is a nonzero constant. 

Then $II(Q) (\zeta \odot \zeta)=0$ if and only if  $b_{m_k, k}=0$, and this concludes the proof. 
\end{proof}

\begin{theorem}
\label{cliff>d}
Let $C$ be a smooth curve of genus $g$,  take a positive  integer $d < Cliff(C)$. Then there are no asymptotic directions $\zeta \in H^1(T_C)$ of rank $d$. 

\end{theorem}
\begin{proof}
Take an element $\zeta \in H^1(T_C)$ of rank $d$. By Theorem \ref{cliff-rank}, $\zeta$ is a linear combination of Schiffer variations supported on a divisor $D$ of degree $d$, so we conclude by Theorem \ref{no-schiffer-cliff}. 
\end{proof}


Notice that, since the Clifford index of  the general curve of genus $g$ is  $\lfloor\frac{g-1}{2}\rfloor$,  Theorem \ref{cliff-rank}(1),  and Theorem  \ref{cliff>d} imply the following 

\begin{corollary}
Let $C$ be  a general curve in ${\mathcal M}_g$. Then
\begin{enumerate}
\item all tangent directions $\zeta$ of rank $d$, with  $0<d< \lfloor\frac{g-1}{2}\rfloor$,  are linear combinations of Schiffer variations supported on an effective divisor of degree $d$. 
\item There are no asymptotic directions of rank $d$, with  $0<d< \lfloor\frac{g-1}{2}\rfloor$. 
\end{enumerate}

\end{corollary}

\begin{remark}
For the general curve of genus $g$ we characterise the image of the natural map ${\mathbb P}T_{C_d}\to {\mathbb P}(H^1(T_C))$
as the locus of deformations of rank at most $d$, for all $0<d< \lfloor\frac{g-1}{2}\rfloor$.  

Moreover for such values of $d$ the base locus of the linear system of quadrics $II(I_2)$ in ${\mathbb P}H^1(T_C)$, does not contain any point $[\zeta]$ with with $Rank(\zeta) =d$. 
\end{remark}

\section{Double-split deformations}
We will give some computation of the second fundamental form along some tangent directions.
Let $C$ be a curve of genus $g$ we assume $C$  non hyperelliptic.
Let  $L$ be a line bundle on $C$ and $M=K_C \otimes L^{\vee}$. We assume that $h^0(L)>1$ and $h^0(M)>1.$ 
We consider the map $$\phi:\bigwedge^2 H^0(L)\otimes \bigwedge^2 H^0(M)\to I_2,$$

\begin{equation}
\label{qua}
\phi((s_1\wedge s_2)\otimes (\tau_1\wedge \tau_2))= (s_1\tau_1)\odot(s_2\tau_2)-(s_1\tau_2)\odot(s_2\tau_1).\end{equation}
\begin{remark} (see \cite{AM})
\label{triqua}
If $C$ is a trigonal curve of genus $g \geq 4$ and $L$ is a $g^1_3$ the map $\phi$ is an isomorphism. 
\end{remark}

Set $N: = K_C \otimes L^{-2}$ and assume $|L|$ is base point free, that $h^0(L) \geq 2$, and $h^0(N) \geq 1$. 
Fix a nonzero section $t \in H^0(N)$, a base point free pencil $\langle s_1, s_2 \rangle$ in $H^0(L)$, such that $s_1$, and $s_2$ have  zeros disjoint from the zeros of $t$. Consider the quadric 
$$Q = (s^2_1t)\odot(s_2^2t)-(s_1s_2 t)\odot(s_1 s_2 t) \in I_2,$$
set $\omega := s_1s_2t$, denote by $D$ the zero locus of $s_1$, and take $\zeta:= [ \debar(\frac{\rho_D}{\omega})] \in H^1(T_C)$. 

\begin{definition}
\label{double_split}
A nonzero deformation $\zeta \in H^1(T_C)$ is said to be double-split if it is split with $\omega = s_1 \tau$, $s_1 \in H^0(L)$, $\tau  \in H^0(K_C \otimes L^{\vee})$ and $ \tau = s_2 t$ with $s_2 \in H^0(L)$ and $t \in H^0(K_C \otimes L^{-2})$. 
\end{definition}

\begin{proposition}
\label{applicazione}
With the above notation, if $\zeta$ is double-split, and $Q = (s^2_1t)\odot(s_2^2t)-(s_1s_2 t)\odot(s_1 s_2 t),$ 
we have 
$$II(Q)(\zeta \odot \zeta) = 2 \pi i (deg(L))  \neq 0.$$
So $\zeta$ is not an asymptotic direction. 
\end{proposition}

\begin{proof}
We set $\omega_1 := s_1^2t$, $\omega_2:= s_2^2 t$, $\omega_3 = \omega_4 := \omega = s_1s_2t$, $x := \frac{s_1}{s_2}$.  Then with the notation of Proposition \ref{w}, we have  $g_3  = g_4 =1$, $\rho_D g_1 = \rho_D \frac{s_1}{s_2} = \rho_D x$ is ${\mathcal C}^{\infty}$, since $\rho_D$ is identically zero on the zeros of $s_2$.  Then we can apply Proposition \ref{w} and we get 
$$II(Q) (\zeta \odot \zeta) = - 2 \pi i \sum_{p \in Supp(D)} Res_p(g_1 dg_2 -g_3 dg_4) =  - 2 \pi i \sum_{p \in Supp(D)} Res_p(g_1 dg_2) = $$
$$=-  2 \pi i \sum_{p \in Supp(D)} Res_p(x d(\frac{1}{x})) =   2 \pi i \sum_{p \in Supp(D)} Res_p(\frac{dx}{x})= 2 \pi i   (deg(D) )=  2 \pi i (deg(L)).$$

\end{proof}

\section{Rank $d=Cliff(C)$ deformations}

In this section we will consider infinitesimal deformations of rank $d= Cliff(C)$ and we will give sufficient conditions under which $\zeta$ is not asymptotic. 

Take a smooth non hyperelliptic curve $C$ of genus $g \geq 4$ and an element $\zeta \in H^1(T_C)$, corresponding to the class of an extension as in \eqref{extension}. 


Taking gobal sections in \eqref{extension}, if $\zeta$ has rank $d$, we get $h^0(E) = g-d+1$. 



\begin{remark}
\label{bau}
Assume $C$ is not hyperelliptic. Recall that if $\zeta \in H^1(T_C)$ is a non trivial deformation of rank $d= Cliff(C) < \frac{g-1}{2}$ and $E$ is not globally generated, then by Theorem \ref{cliff-rank},  $\zeta$ is a linear combination of  Schiffer variations supported on a degree $d$ effective divisor. 
\end{remark}

\begin{theorem}
\label{cliff=d}
Assume $\zeta \in H^1(T_C)$ is an infinitesimal deformation of positive rank $d$ on a curve $C$ of Clifford index $d \neq \frac{g-1}{2}$ and $g \geq 5$. Assume moreover that $Cliff(C) = Cliff(L) = gon(C) -2$, where $L$ is a $g^1_{d+2}$ on $C$. 
Assume that $\zeta$ is not a linear combination of  Schiffer variations supported on a degree $d$ effective divisor. 

\begin{itemize}
\item If for every such $L$, $h^0(L^{\otimes 2}) = 3$, then $\zeta$ is split. 
\item  If for every such $L$, $h^0(L^{\otimes 2}) = 3$ and $h^0(L^{\otimes 3}) = 4,$ then $\zeta$ is double split and it is not an asymptotic direction. 

\end{itemize}

\end{theorem}

\begin{proof}
By the assumptions and by Remark \ref{bau}, we know that $E$ is globally generated. Take two generic points $p, q \in C$, then since $g \geq 5$, there exists a nontrivial holomorphic section $\sigma \in H^0(E(-p-q))$. 

Tensoring \eqref{extension} by ${\mathcal O}_C(-p-q)$, we obtain an extension 
\begin{equation}
\label{extension-new}
0 \to {\mathcal O}_C(-p-q) \to E' \to K_C(-p-q) \to 0,
\end{equation}
where $E' = E \otimes {\mathcal O}_C(-p-q)$, so $h^0(E') \geq g-d-3 >0$. 
By the theorem of Segre-Nagata and Ghione, (\cite{LazII} p. 84),  there exists a subline bundle $A$ of $E'$ such that $deg(A) \geq \frac{g-5}{2}$. We claim that $h^0(A) \geq 1$. In fact, if $h^0(A) =0$, by Riemann Roch we would have $h^0(K_C\otimes A^{\vee}) = g-1 - deg(A) \leq \frac{g+3}{2}$. So, since $p,q$ are generic points, $h^0(K_C\otimes A^{\vee}( -2p-2q)) \leq \frac{g-5}{2}$. On the other hand,  up to saturation of $A$, so that the quotient $E'/A$ is torsion free, we have the exact sequence
$$ 0 \to A \to E' \to K_C\otimes A^{\vee}( -2p-2q) \to 0,$$
and $h^0(K_C\otimes A^{\vee}( -2p-2q)) \geq  h^0(E') \geq g-d -3$. So we get $d \geq \frac{g-1}{2}$, which contradicts our assumption. So $h^0(A) \geq 1$. We have the following diagram:

$$
\xymatrix{
& & 0\ar[d]&\\
             &  & A(p+q) \ar[d] & & \\
    0 \ar[r]& {\mathcal O}_C \ar[r] & E\ar[d]\ar[r]& K_C\ar[r] & 0 \\
     & & K_C\otimes A^{\vee}( -p-q)\ar[d]& \\
     &  & 0& \\
    &&&& }           
$$
We have $h^0(A(p+q)) \geq 1$ and $deg(A(p+q)) \geq \frac{g-1}{2}$. 
If $h^0(A(p+q)) =1$, then $h^0(K_C\otimes A^{\vee}( -p-q))= g-2 - deg(A) \leq \frac{g+1}{2}$, but $h^0(K_C\otimes A^{\vee}( -p-q)) \geq h^0(E) -1 =g-d$, so we get $d \geq \frac{g-1}{2}$, a contradiction. So $h^0(A(p+q))  \geq 2$ and $h^0(K_C\otimes A^{\vee}( -p-q)) \geq 2$, since $E$ is globally generated, hence also $K_C\otimes A^{\vee}( -p-q)$ is globally generated and non trivial. 

So $A(p+q)$ contributes to the Clifford index and we have: 
$$h^0(A(p+q)) - h^0(K_C\otimes A^{\vee}( -p-q)) = deg(A) -g+3,$$
$$h^0(A(p+q)) + h^0(K_C\otimes A^{\vee}( -p-q))\geq g-d +1.$$
So summing up we get 
$$2 h^0(A(p+q)) \geq deg(A) -d +4,$$
hence $d \leq Cliff(A(p+q)) \leq d$, and either $A(p+q) =L$, or $A(p+q) =K_C \otimes L^{\vee}$, where $L$ is any line bundle as in the statement. But $A(p+q) \neq L$, since $p,q$ are general and $A$ is effective. So $A(p+q) =K_C \otimes L^{\vee}$, and  the vertical exact sequence induces an exact sequence on global sections. So it corresponds to a class of an extension $\epsilon \in H^1(K_C \otimes L^{-2})$ that induces the zero map $H^0(L) \to H^1(K_C \otimes L^{\vee})$. The map $H^1(K_C \otimes L^{-2}) \cong H^0(L)^{\vee} \rightarrow Hom(H^0(L), H^1(K_C \otimes L^{\vee})) \cong H^0(L)^{\vee} \otimes H^0(L)^{\vee}$, $v \mapsto \cup v$, is the dual of the multiplication map $ H^0(L) \otimes  H^0(L) \rightarrow  H^0(L^{\otimes 2})$ and $\epsilon$ is in the kernel of this map. If $h^0( L^{\otimes 2}) =3$, by the base point free pencil trick, the multiplication map $Sym^2 H^0(L) \to H^0(L^{\otimes 2})$ is surjective, so $\epsilon$ must be zero. Hence $E = (K_C \otimes L^{\vee}) \oplus L$ and $\zeta$ is split.

By Proposition \ref{prop_split},   the diagram becomes:

\begin{equation}
\label{spezzato}
\xymatrix{
& & 0\ar[d]&\\
             &  & K_C \otimes L^{\vee} \ar[d] & & \\
    0 \ar[r]& {\mathcal O}_C \ar[r]^{(\tau,s) \ \ \ \  \ \ } &  (K_C \otimes L^{\vee} )\oplus L \ar[d]\ar[r]^{\ \ \ \ \ s-\tau}& K_C\ar[r] & 0 \\
     & & L\ar[d]& \\
     &  & 0& \\
    &&&& }           
\end{equation}
where $s \in H^0(L)$ and $\tau \in H^0( K_C \otimes L^{\vee})$ have disjoint zero loci. 
Then $ker (\cup \zeta) = s \cdot H^0( K_C \otimes L^{\vee}) + \tau \cdot H^0(L),$ which has dimension $g-d$. 

Assume now $h^0(L^{\otimes 3}) = 4$, then the multiplication map 

$$m: H^0(L)  \otimes H^0(K_C\otimes L^{-2}) \to H^0(K_C \otimes L^{\vee})$$
is surjective, since $ker(m) = H^0(K_C \otimes L^{-3})$ by the base point free pencil trick. So $\tau = s t + s' t'$, with $s' \in H^0(L)$, $t, t' \in H^0(K_C \otimes L^{-2})$. 
We claim that if we take $\omega' = s \tau'$, with $\tau' = s' t'$, then $[ \debar(\frac{\rho_D}{\omega'})]  =  [ \debar(\frac{\rho_D}{\omega})] = \zeta$, where $D$ is the zero divisor of $s$ and $\omega = s \tau$. 

In fact $[ \debar(\frac{\rho_D}{\omega}) - \debar(\frac{\rho_D}{\omega'})] =  [\debar(\rho_D (\frac{-t}{s't'(st + s't')}))] = 0$, since $\rho_D (\frac{-t}{s't'(st + s't')})$ is  a ${\mathcal C}^{\infty}$ vector field, as $\frac{-t}{s't'(st + s't')}$ has no poles on $D$ by construction.

So $\zeta$ is double-split and we apply Proposition \ref{applicazione} to conclude.


\end{proof}

\section{Rank one deformations}

In this section we consider infinitesimal deformations $\zeta$ of rank 1 on a non hyperelliptic curve $C$ of genus $g \geq 4$, and we ask whether these can be asymptotic. In  Theorem \ref{cliff>d} we proved that if  $Cliff(C) >1$, there are no asymptotic directions of rank 1. So we will assume $Cliff(C) =1$, hence $C$ is either trigonal or isomorphic to a smooth plane quintic. 

We will show that if $C$ is trigonal of genus $g \geq 8$, or $g=6,7$ with Maroni degree 2, then there do not exist asymptotic directions of rank 1, except for the Schiffer variations at a ramification point of the $g^1_3$, which are asymptotic. 

We will also prove that on a smooth plane quintic there are no rank one asymptotic directions.

In section \ref{Maroni} we will show that if $g=5$, or $g=6,7$ and Maroni degree 1, there are examples of curves admitting asymptotic directions different from Schiffer variations at  ramification points, and we will describe these loci.

Take a smooth non hyperelliptic curve $C$ of genus $g \geq 4$ and an element $\zeta \in H^1(T_C)$. The cup product 
$\cup \zeta: H^0(K_C) \to H^1({\mathcal O}_C),$
corresponds to an element $\gamma_{\zeta}$ in $S^2 H^1({\mathcal O}_C)$. 
We have the following exact sequence given by the differential of the period map (dual to the multiplication map): 

$$ 0 \to H^1(T_C)\stackrel{\gamma} \to S^2 H^1({\mathcal O}_C) \to I_2^{\vee} \to 0, $$
where $\gamma(\zeta) = \gamma_{\zeta}$.

\begin{remark}
\label{quadrica}
The rank of $\zeta$ is equal to the rank of the quadric $\gamma_{\zeta}$.
\end{remark}

So an element $\zeta \in H^1(T_C)$ has rank one if and only if $[\gamma_{\zeta}] \in {\mathbb P}  (S^2 H^1({\mathcal O}_C)) $ lies in the image the Veronese map 
$$ {\mathbb  P}  (H^1({\mathcal O}_C)) \to {\mathbb P}  (S^2 H^1({\mathcal O}_C)).$$

Recall that by \cite{griffiths} if $C$ is non hyperelliptic, non trigonal and not isomorphic to a smooth plane quintic, the only rank one elements $[\zeta] \in {\mathbb P}(H^1(T_C))$ are given by the classes of the Schiffer variations $[\xi_p]$, that are the points of the bicanonical curve. If $C$ is trigonal the rank one elements correspond to the Veronese image of the ruled surface containing the canonical curve, and if $C$ is a smooth plane quintic they correspond to the points of the Veronese image of the Veronese surface in $ {\mathbb  P}  (H^1({\mathcal O}_C))\cong {\mathbb P}^5$.

Since $\zeta$ has rank 1, in the extension \eqref{extension} corresponding to $\zeta$,  we get $h^0(E) = g$.

\begin{lemma}
\label{lem_schiffer}
Assume $C$ is not hyperelliptic, $\zeta \in H^1(T_C)$ a rank one deformation.  If $E$ is not globally generated, then $\zeta = \xi_p$ is a Schiffer at a point $p$.  
If $C$ is non trigonal, then $\forall p \in C$, $\xi_p$ is not an asymptotic direction. 
If $C$ is trigonal and $g \geq 6$, then $\xi_p$ is an asymptotic direction if and only if $p$ is a ramification point of the $g^1_3$. 

\end{lemma}

\begin{proof}
Assume $E$ is not globally generated, we conclude by Theorem \ref{cliff-rank} that $\zeta = \xi_p$ is a Schiffer variation at a point $p$.


If $g \geq 5$, $C$ is not hyperelliptic, and not trigonal, then  for any point $p \in C$, $\xi_p$ is not an asymptotic direction. In fact, in \cite{cpt}  (see also \cite[Theorem 2.2]{cfg})  it is proven that for any $Q \in I_2(K_C)$, we have $$II(Q)(\xi_p \odot \xi_p) = -2 \pi i \mu_2(Q)(p),$$
and by \cite[Theorem 6.1]{cf-Michigan} we know that if $C$ is not hyperelliptic, and not trigonal of genus $g \geq 5$, for any $p \in C$, there exists a quadric $Q \in I_2(K_C)$ such that $\mu_2(Q)(p) \neq 0$.

Assume that $C$ is trigonal not hyperelliptic and $g \geq 5$, then  a basis for $I_2(K_C)$ is given by the quadrics of rank at most 4 defined as in \eqref{qua} where $L$ is the $g^1_3$ (see Remark \ref{triqua}).

For such quadrics we have 
$$II(Q)( \xi_p \odot \xi_p) = -2 \pi i \mu_2(Q)(p)= -2 \pi i \mu_{1,L}(s_1 \wedge s_2)(p) \mu_{1, K_C(-L)}(\tau_1 \wedge \tau_2)(p),$$
 (see \cite{cpt} for the first equality, and \cite[Lemma 2.2]{cf-Michigan} for the second one). 

If $C$ is trigonal and $p$ is a ramification point of the $g^1_3$, $\rho(Q)( \xi_p \odot \xi_p) =0$, since $\mu_{1,L}(s_1 \wedge s_2)(p)=0$. 

Assume $C$ trigonal and $p$ not a ramification point of the $g^1_3$, then $h^0(K_C\otimes L^{\vee}(-p)) = h^0(K_C\otimes L^{\vee})-1$, otherwise $|L(p)|$ is a $g^2_4$ and $C$ would be  hyperelliptic. We claim that if $g \geq 6$,  $h^0(K_C\otimes L^{\vee}(-2p)) = h^0(K_C\otimes L^{\vee}(-p))-1$. In fact, otherwise $|L(2p)|$ would be a $g^2_5$, so it would give a  map $C \to {\mathbb P}^2$ of degree 1, hence $g \leq 6$. So $g=6$ and $C$ would be  a smooth plane quintic, hence not trigonal. So if we take a pencil $\langle \tau_1, \tau_2 \rangle \subset H^0(K_C\otimes L^{\vee})$ such that  $\langle \tau_1, \tau_2 \rangle \cap H^0(K_C\otimes L^{\vee}(-2p)) =\{0\}$, we get 
$$II(Q)( \xi_p \odot \xi_p) = 2 \pi i \mu_2(Q)(p)= 2 \pi i \mu_{1,L}(s_1 \wedge s_2)(p) \mu_{1, K_C(-L)}(\tau_1 \wedge \tau_2)(p) \neq 0,$$
since $p$ is not a ramification point for $|L|$, nor for the pencil $|\langle \tau_1, \tau_2 \rangle |$. 


\end{proof}



\begin{definition}  
\label{maroni}
Let $C$ be a trigonal (non hyperelliptic) curve  of genus $g \geq 5$ and let $L$ be the line bundle of degree 3 computing the unique trigonal series.  The Maroni degree
$k \in {\mathbb N}$ of $C$ can be characterised as the unique number such that
$$h^0(C, L^{\otimes k+1}) = k + 2,  \ h^0(C, L^{\otimes k+2}) > k + 3.$$
\end{definition}
The following bounds on $k$ have been established by Maroni (\cite{maroni})
$$\frac{g - 4}{3} \leq  k \leq \frac{g - 2}{2}.$$

Hence if $g\geq 5$, we can have trigonal curves with  $k=1$ only if $g=5,6,7$. This means that $h^0(L^{\otimes 2}) =3$, and $h^0(L^{\otimes 3}) =5$. For $g=5$, we have $k=1$, while for $g=6,7$ the general curve has $k =2$. We will say that a trigonal curve of genus $g =6,7$ is Maroni special if $k=1$. 

Notice that if $g \geq 8$, we always have  $k \geq 2$.

\begin{theorem}
\label{rank1}
If $C$ is trigonal (non hyperelliptic) of genus of genus $g \geq 8$, or $C$ is trigonal (non hyperelliptic) of genus $g =6,7$ and $k =2$,  then the rank one asymptotic directions are exactly the Schiffer variations in the ramification points of the $g^1_3$. 
\end{theorem}

\begin{proof}
Take $\zeta$ a rank 1 infinitesimal deformation and $E$ the rank 2 vector bundle in the corresponding extension \eqref{extension}. 
If $E$ is not globally generated, then by Lemma \ref{lem_schiffer}, $\zeta = \xi_p$ is a Schiffer at a point $p$, and  it is an asymptotic direction if and only if $p$ is a ramification point of the $g^1_3$.

So assume $E$ is globally generated and take $L$ the $g^1_3$. Then, since $C$ is not hyperelliptic, $h^0(L^{\otimes 2}) =3$ and by assumption $h^0(L^{\otimes 3}) = 4$. So by the proof of Theorem \ref{cliff=d},  $\zeta$ is not asymptotic.

\end{proof}

Let us now assume that $C$ is a smooth plane quintic and take $L = \mathcal O_C(1)$  the $g^2_5$, then $L^{\otimes 2} = K_C$.

\begin{theorem}
\label{quintic} 
On a smooth plane quintic there are no rank one asymptotic directions. 

\end{theorem}

\begin{proof}
By the discussion following Remark \ref{quadrica},  the deformations of rank one correspond to the points of ${\mathbb P}^2$, hence they are the intersections of two lines. 
Since  the Schiffer variations are not asymptotic directions (see Lemma \ref{lem_schiffer}), we consider only rank one deformations corresponding to points $ p \in {\mathbb P}^2 \setminus C$.

Choose two lines $l_1, l_2$ in ${\mathbb P}^2$ passing through $p$ that intersect $C$ transversally. Denote by $s_1, s_2$ the corresponding sections in $H^0(L)$. 

Since $L^{\otimes 2} = K_C$, we have $H^0(K_C \otimes L^{- 2}) = H^0({\mathcal O}_C)$, and we choose $t = 1 \in  H^0({\mathcal O}_C)$. 
Choose  $\omega = s_1s_2$, set $D = div(s_1)$.  Then the element $\zeta = [ \debar (\frac{\rho_D}{\omega})]$ is double-split of  rank one, by Proposition \ref{prop_split}. Hence by  Proposition \ref{applicazione}, we have $\rho(Q) (\zeta \odot \zeta) =  (2 \pi i) deg(L) =  10 \pi i,$
so $\zeta$ is not asymptotic. 

\end{proof}
\section{Rank 2 deformations}

In this section we will study rank 2 deformations and  the condition to be asymptotic. 
Let $\zeta \in H^1(T_C)$ be a rank 2 deformation. If $Cliff(C) >2$, by Theorem \ref{cliff>d}, we know that $\zeta$ is not asymptotic. 

So assume $Cliff(C) =2$, hence either $C$ is tetragonal, or it is a smooth plane sextic. 

\begin{theorem}
\label{rank2}
Assume $C$ is a tetragonal curve of genus at least 16 and not a double cover of a curve of genus 1 or 2.  If a deformation $\zeta$ of rank 2  is not a linear combination of Schiffer variations supported on a degree 2 effective divisor, then $\zeta$ is not asymptotic.  
\end{theorem}

\begin{proof}
Under our assumptions,  the rank 2 bundle $E$ in the extension \eqref{extension} associated with the rank 2 deformation $\zeta$  is globally generated (see Remark \ref{bau}). 
We claim that  for every $g^1_4$, $L$,  we have $h^0(L^{\otimes 2}) =3$, $h^0(L^{\otimes 3}) =4$, hence we conclude applying the proof of Theorem \ref{cliff=d}.

In fact, if $h^0(L^{\otimes 2}) \geq 4$, $C$ has a $g^3_8$. This can't give a birational map, by the Castelnuovo bound. The map cannot have degree $4$ otherwise the image would be a conic, hence contained in a plane. So  the map has degree 2 and the image has degree 4, so it is either rational or a genus 1 curve. Since it is not hyperelliptic, nor bielliptic by assumption, we have $h^0(L^{\otimes 2}) =3$. 

If $h^0(L^{\otimes 3})\geq 5$, $C$ has a $g^4_{12}$. Again the induced map can't be birational by  the Castelnuovo bound. If it has degree 3, since it is non degenerate, the image is a rational normal curve of degree 4, hence the curve $C$ is trigonal, which contradicts our assumption on $Cliff(C)$. If the map has degree 2, the image is non degenerate of degree 6, hence  by Clifford's theorem it is non special, hence  by Riemann Roch it is a curve  of genus 1 or 2, a  contradiction. 



\end{proof}

We will now consider the case where $\zeta$ is a linear combination of two Schiffer variations. We have the following 
\begin{theorem}
\label{bielliptic}
On any bielliptic curve of genus at least 5 there exist linear combinations of two Schiffer variations that are asymptotic  of rank 2. 

\end{theorem}
\begin{proof}
Assume  the curve $C$ is bielliptic of genus at least 5, and let $\pi: C \to E$ be the $2:1$ map to a genus 1 curve $E$. By the Castelnuovo-Severi   inequality  (see e.g. \cite[Theorem 3.5]{accola}), a bielliptic curve of genus at least 5 is not hyperelliptic, nor trigonal.

 Denote by $\sigma$ the bielliptic involution. Consider the curve $\Gamma$ in  the surface $S =C\times C$ given by the graph of $\sigma$, $\Gamma : = \{ (p, \sigma(p)) \ | \ p \in C\}$. 
Consider the form $\hat{\eta} \in H^0(K_S(2 \Delta))$ introduced in \cite{cfg}. 
Denote by $Z(\hat{\eta})$ its zero locus. Clearly $Z(\hat{\eta})$ intersects the curve $\Gamma$ outside the diagonal. Take a point $(p, \sigma(p)) \in Z(\hat{\eta}) \cap \Gamma$, $p \neq \sigma(p)$. 
Thus,  for any quadric $Q \in I_2$, by  \cite{cfg} we have 
$$II(Q)(\xi_p \odot \xi_{\sigma(p)}) = - 4 \pi i Q(p, \sigma(p)) \hat{\eta}(p, \sigma(p)) = 0.$$
Take $\zeta = a \xi_p + b \xi_{\sigma(p)}$, then by  \cite{cfg} we have 
$$II(Q)(\zeta \odot \zeta) = a^2 II(Q)(\xi_p \odot \xi_p) + b^2 II(Q)(\xi_{\sigma(p)} \odot \xi_{\sigma(p)}). $$
Notice that any quadric $Q \in I_2$ is $\sigma$-invariant. In fact, $H^0(C, K_C) \cong H^0(E, K_E) \oplus H^0(C, K_C)^-$, where $H^0(C, K_C)^-$ denotes the anti-invariant subspace by the action of $\sigma$ and $H^0(E, K_E) \cong H^0(C, K_C)^+ $ is the invariant subspace. 
Hence  $\operatorname{dim} \big( Sym^2 H^0(C,K_C) \big)^- =g-1$. Moreover  we have $\operatorname{dim}H^0(C,K_C^{\otimes 2})^+= 2g-2$, thus $\operatorname{dim}H^0(C,K_C^{\otimes 2})^- =  g-1$. Since the multiplication map is $\sigma$-equivariant,  $I_2(K_C)^- = (0)$, so  $I_2(K_C) = I_2(K_C)^+$.
Then we have 
$$II(Q)(\zeta \odot \zeta) =  (a^2 + b^2)  II(Q)(\xi_p \odot \xi_p),$$
since $Q $ is $\sigma$-invariant, $II$ is $\sigma$-equivariant and $\sigma_*(\xi_p) = \xi_{\sigma(p)}$. 
Hence, if $a^2 +b^2 = 0$, $\zeta$ is asymptotic. On the other hand, notice that if $a^2 + b^2 \neq 0$, there exists a quadric $Q$ such that $II(Q)(\xi_p \odot \xi_p) \neq 0$, so $\zeta$ is not asymptotic (see Lemma \ref{lem_schiffer}).

\end{proof}

\begin{remark}
The asymptotic directions found in Theorem \ref{bielliptic} are linear combinations of Schiffer variations supported on a divisor $D = 2p + 2q$ of a $g^1_4$ on the bielliptic curve $C$.

\end{remark} 

\begin{proof}
For any point $(p, \sigma(p))$ of the curve $\Gamma$ in the proof of Theorem \ref{bielliptic}, the line bundle ${\mathcal O}_C(2p + 2 \sigma(p))$ is a $g^1_4$ on $C$. In fact, for $p \in C$, denote by $y = \pi(p) = \pi(\sigma(p)) \in E$, and take a $2:1$ cover $\phi: E \to {\mathbb P}^1$ such that $y$ is a ramification point of $\phi$. Then $\phi(y)$ is a critical value for the degree $4$ map $\psi:= \phi \circ \pi$, such that $\psi^*(\phi(y)) = 2p + 2 \sigma(p)$. 
\end{proof}
\
\

Assume now that $C$ is a smooth plane sextic. 
We will describe the rank 2 deformations on $C$. 
Set $V: = H^0({\mathcal O}_{{\mathbb P}^2}(3) ) \cong H^0(C, K_C)$ and consider the Grassmannian $G(8,V)$ of linear subspaces of $V$ of codimension 2. Denote by $Y_k = \{[\zeta] \in {\mathbb P}H^1(T_C) \ | \ rank(\zeta) \leq k\}$ and consider the map 
$$\chi: Y_2 \setminus Y_1 \to G(8,V), $$
$$[\zeta] \mapsto ker(\zeta).$$
Denote by $Sec(C) \subset Y_2$ the linear combination of Schiffer variations at two points in $C$. 

Let $Z \in Hilb^2({\mathbb P}^2)$ be a length 2 scheme and denote by ${\mathcal I}_Z$ the ideal sheaf of $Z$. Then $W_Z := H^0({\mathcal I}_Z(3)) \subset V$ has codimension 2. If $Z = p+q$, $p \neq q$, then $W_Z$ is the space of cubics passing through $p$ and $q$, while if $Z=(p, v)$, where $v \neq 0$ is a tangent vector at $p$, the elements in $W_Z$ are the cubics passing through $p$ and tangent at $v$. Consider  the injective map 
$$\gamma:  Hilb^2({\mathbb P}^2) \to G(8,V), \ Z \mapsto W_Z.$$ Denote by $Hilb^2({\mathbb P}^2)(C)$ the divisor of schemes having support intersecting $C$ and $U_C := Hilb^2({\mathbb P}^2) \setminus Hilb^2({\mathbb P}^2)(C)$. We have the following 

\begin{proposition}
\label{sextic}
We have 
\begin{enumerate}
\item $\chi(Y_2 \setminus Y_1) \subset \gamma(Hilb^2({\mathbb P}^2))$,
\item $\chi$ induces a bijection between $Y_2 \setminus Sec(C) $ and $U_C$. 
\end{enumerate}

\end{proposition}

\begin{proof} 
Take $[\zeta] \in Y_2 \setminus Y_1$, and $W = \chi([\zeta]) = ker(\zeta) \subset V$. Denote by  
$$ev: W \otimes {\mathcal O}_{{\mathbb P}^2}   \to  {\mathcal O}_{{\mathbb P}^2}(3)$$
the evaluation map and tensor it by ${\mathcal O}_{{\mathbb P}^2} (-3)$. Then the cokernel of this map is a sheaf ${\mathcal O}_{Z}$ and we claim that it  has length $2$.

In fact, denote by $m_i: W \otimes H^0({\mathcal O}_{{\mathbb P}^2}(i) )  \to  H^0({\mathcal O}_{{\mathbb P}^2}(3+i))$ the multiplication map and by $c_i:= codim(Im (m_i)), H^0({\mathcal O}_{{\mathbb P}^2}(3+i))$.

Then by a Theorem of Macaulay (see \cite[Theorem 2]{Green_Restriction}), we have $c_1 \leq 2$. Gotzmann's Persistence Theorem (see  \cite{Green_Restriction}) says that if $c_1 =2$, then $c_i =2$ for all $i \geq 1$. From this one easily proves that ${\mathcal O}_{Z}$ has length 2. 

This shows that $\chi([\zeta]) = \gamma(Z)$ and $(1)$ is proven. 
Since $\gamma$ is injective, we can consider the composition 
$$\beta:= \gamma^{-1} \circ \chi: Y_2 \setminus Y_1 \to Hilb^2({\mathbb P}^2).$$

We claim that $\beta^{-1}(Hilb^2({\mathbb P}^2)(C)) \subset Sec(C)$. 

Assume by contradiction that there exists $[\zeta] \not \in Sec(C)$ such that $\beta([\zeta]) \in Hilb^2({\mathbb P}^2)(C)$. Then $\beta([\zeta]) = p+q$, where $p \in C$ and $q \not \in C$ and $\chi([\zeta]) = ker(\zeta) \cong \{ s \in H^0({\mathcal O}_{{\mathbb P}^2}(3)) \ | \ s(p) = s(q) =0\}$. 

The image of $m_3$ is given by the sextics vanishing at $p$ and $q$. So, restricting to $C$, we obtain that the image of the map 
$\mu: W \otimes H^0(C, K_C) \rightarrow H^0(C, K_C^{\otimes 2})$ 
 is $H^0(K_C^{\otimes 2}(-p))$. Hence we get $\zeta = \xi_p \in Y_1$. 

So we can consider the restriction $\beta'$ of $\beta$ to $Y_2 \setminus Sec(C)$. 
This gives a map 
$$\beta' : Y_2 \setminus Sec(C) \to U_C.$$

We want to show that $\beta'$ is bijective. 

Assume $\beta'([\eta]) = \beta'([\zeta]) = Z$. So $W = ker(\eta) = ker(\zeta)$, hence $W \subset ker(a\eta + b \zeta)$, for every linear combination. Note that the rank of $a\eta + b \zeta$ cannot be 1, otherwise $W \subset H^0(K_C(-p))$ for some $p \in C$, and so $p \in Supp(Z)$, a contradiction since $Z \in U_C$. On the other hand, letting $L = H^0(K_C)/W$, then $a \eta + b \zeta$ define symmetric forms on $L$, which has dimension 2, so there exists a linear combination $a \eta + b \zeta$ for which  the rank of the corresponding quadric drops. Hence it is zero and  $[\zeta] = [\eta] \in Y_2$, thus  $\beta'$ is injective. 

To prove that $\beta'$ is surjective, fix $Z \in U_C$ and realise $Z$ as a complete intersection of a line $l$ and a conic $\Gamma$ (tangent if $Z$ is not reduced). Using the equations of $l$ and $\Gamma$ we obtain the exact sequence

$$ 0 \to {\mathcal O}_{{\mathbb P}^2}(-3) \to {\mathcal O}_{{\mathbb P}^2}(-1) \oplus {\mathcal O}_{{\mathbb P}^2}(-2 ) \to {\mathcal I}_Z  \to 0$$
and tensoring by ${\mathcal O}_{{\mathbb P}^2}(3)$ we have 
$$ 0 \to {\mathcal O}_{{\mathbb P}^2} \to {\mathcal O}_{{\mathbb P}^2}(2) \oplus {\mathcal O}_{{\mathbb P}^2}(1 ) \to {\mathcal I}_Z(3)  \to 0.$$
Restricting to $C$ and using the fact that $Z \cap C = \emptyset$, we get the exact sequence
$$ 0 \to {\mathcal O}_{C} \to {\mathcal O}_{C}(2) \oplus {\mathcal O}_{C}(1) \to K_C  \to 0,$$
and its extension class gives a non trivial class $\zeta$ such that $ker (\zeta) = H^0( {\mathcal I}_Z(3)_{|C})$. So $\beta'([\zeta]) = Z$. 
 \end{proof}

\begin{remark}
Notice that from Proposition \ref{sextic} (2), it follows that the restriction of $\chi$ to $Y_2 \setminus Sec(C)$ is injective. Hence for these infinitesimal deformations $\zeta$, $ker(\zeta)$ determines $\zeta$. 

\end{remark}
We show the following

\begin{theorem}
\label{sextic1}
On a smooth plane sextic there are no asymptotic directions of rank 2. 

\end{theorem}

\begin{proof}
By Proposition \ref{sextic} we have shown that a rank 2 deformation $\zeta$ that is not a linear combination of Schiffer variations corresponds to a point $Z \in Hilb^2({\mathbb P}^2)$ whose support does not intersect $C$. We have $ker(\zeta) \cong H^0({\mathcal O}_{{\mathbb P}^2}(3) \otimes {\mathcal I}_Z)$ and $\zeta$ is determined by its kernel. 

The deformation $\zeta$ corresponds to such a length 2 scheme $Z$, that is the intersection of a line $l_1$ and a conic $\Gamma$, which are tangent in a point if $Z$ is not reduced. So $ker(\zeta) \cong  H^0({\mathcal O}_{{\mathbb P}^2}(3) \otimes {\mathcal I}_Z) = l \cdot H^0({\mathcal O}_{{\mathbb P}^2}(2)) + \Gamma \cdot H^0({\mathcal O}_{{\mathbb P}^2}(1))$. The scheme  $Z$ is either supported in two distinct  points $p$, $q$, or it is given by $p$ and a tangent direction $v$. So the line $l_1$ is uniquely determined by $Z$, while we can choose $\Gamma$ as the union of two lines $l_2, l_3$ passing through $p$ and $q$, or passing through $p$ if $Z$ is not reduced. 
Denote by $l_{i{|C}} =: s_i \in H^0({\mathcal O}_{C}(1))$, set $\omega = s_1s_2s_3 \in H^0(K_C)$ and $D: = div(s_1)$. Then a Dolbeault representative for $\zeta$ is $ \overline{\partial}(\frac{\rho_D}{\omega})$, where $\rho_D$ is as in Proposition \ref{prop_split}. This can be easily seen, observing that $ ker ([\overline{\partial}(\frac{\rho_D}{\omega})])$ is equal to $ker(\zeta)$. In fact, $ ker ([\overline{\partial}(\frac{\rho_D}{\omega})])$ is equal to $s_1 \cdot H^0({\mathcal O}_C(2)) + s_1 s_2 \cdot H^0({\mathcal O}_C(1))$, as it is easily seen using Proposition  \ref{prop_split}.

 So $\zeta$ is double split, hence it is not asymptotic by Proposition \ref{applicazione}. 

Assume $\zeta$ is a linear combination of two Schiffer variations, $\zeta = a \xi_p + b \xi_q$, with $p, q$ two distinct points in $C$ and not lying on a bitangent of $C$. 

Consider $s_1 \in H^0({\mathcal O}_C(1)(-p-q))$, $s_2 \in  H^0({\mathcal O}_C(1))$ such that $\langle s_1, s_2 \rangle$ is base point free and such that $p$  is not a ramification point of this pencil. Let $s_3 \in H^0({\mathcal O}_C(1))$, with $s_3(q) =0$, $s_3(p) \neq 0$. Set $Q = s_1^2 s_3 \odot s_2^2 s_3 - s_1 s_2s_3 \odot s_1 s_2s_3 \in I_2(K_C)$. Then by \cite[Lemma 2.2]{cf-Michigan} we have 
$$\mu_2(Q) = s_3^2 (\mu_{1, {\mathcal O}_C(1)}(s_1 \wedge s_2))^2 \in H^0(K_C^{\otimes 4}) ,$$
so $\mu_2(Q)(p) \neq 0$, while $\mu_2(Q)(q) =0$ and $Q(p,q) =0$. So by \cite[Thm. 2.2]{cfg}, we have  $II(Q)(\zeta \odot \zeta) =  \pi i a^2 \mu_2(Q)(p) =0$ if and only if $a=0$ and $\zeta$ is a Schiffer variation at $q$, so it is not asymptotic. 

If the line through $p$ and $q$ is bitangent, take $s_2$ the section of $H^0({\mathcal O}_C(1))$ given by a line passing through $p$ and not through $q$, 
$s_3, s_4$ the sections  given by distinct  lines passing through $q$ and not through $p$. Then the quadric $Q = s_2^2 s_4\odot s_3^2 s_4 - s_2 s_3s_4 \odot s_2 s_3s_4 \in I_2(K_C)$, is such that $Q(p,q) =0$, since $s_2s_4(p) = s_2s_4(q) =0$. By  \cite[Lemma 2.2]{cf-Michigan}, we have 
$$\mu_2(Q) = s_4^2 (\mu_{1, {\mathcal O}_C(1)}(s_2 \wedge s_3))^2 \in H^0(K_C^{\otimes 4}) .$$
Hence $\mu_2(Q)(p) \neq 0$, while $\mu_2(Q)(q) =0$. So by \cite[Thm. 2.2]{cfg}, we have $II(Q)(\zeta\odot \zeta) =  -2 \pi i a^2 \mu_2(Q)(p) =0$ if and only if $a=0$ and $\zeta$ is a Schiffer variation at $q$, so it is not asymptotic. 

If $p =q$, $\zeta = a \xi_p + b \xi_p^2$, take the tangent line of $C$ at $p$ and denote by $s_1$ the given section of $H^0({\mathcal O}_C(1))$. Take $s_2$, $s_3$ two other sections of $H^0({\mathcal O}_C(1))$ given by two  lines not passing through $p$. Then consider the quadric $Q =   s_1^2 s_3 \odot s_2^2 s_3 - s_1 s_2s_3 \odot s_1 s_2s_3 \in I_2(K_C)$.  Set $\omega = s_1 s_2 s_3 \in H^0(K_C)$, then $\zeta = [\debar(\rho_p (\frac{az +b}{\omega}))]$ and the computation in Proposition \ref{w} gives $II(Q)(\zeta \odot \zeta) = b^2 Res_p(\frac{s_1}{s_2} d (\frac{s_2}{s_1}))= 0$ if and only if $b =0$, hence $\zeta$ is a Schiffer variation at $p$, so it is not asymptotic.

\end{proof}

\section{Maroni special trigonal curves of  genus $6,7$ and trigonal curves of genus $5$}
\label{Maroni}
\subsection{$g=6$ Maroni special.}
Let $C$ be a (non hyperelliptic)  trigonal curve of genus $6$ and Maroni degree $k=1.$ 
We will show that in this case there can exist asymptotic directions that are not Schiffer variations in the ramification points of the $g^1_3$. We will describe these asymptotic directions. Moreover we give a parametrisation of the locus of the trigonal curves of genus 6 with Maroni degree 1 giving an explicit equation. We also  describe  the sublocus of those trigonal curves admitting such asymptotic directions. 

Denote by $L$ the trigonal linear series. Recall that by the definition \ref{maroni} of the Maroni degree, we have $h^0(L^{\otimes 2}) =3$ and $h^0(L^{\otimes 3}) =5$.
Then, $K_C \otimes L^{-3}$ has degree 1 and by Riemann Roch $h^0(K_C \otimes L^{-3}) = 1$, so $K_C=L^{\otimes 3}\otimes {\mathcal O}_C(q)  $ for a point $q \in C$.  Hence 
$M=K \otimes L^{\vee}=L^{\otimes 2}(q)$, by Riemann Roch  $h^0(M)=4$ and the map $\phi: C\to |M|\cong \bP^3$ is an embedding. 
In fact  $\phi(C)$ is a curve of degree $7$ and it is smooth since otherwise $C$ would have a $g^2_5$. We fix a basis $\{s_1,s_2\}$ of $H^0(L)$  where we assume that $s_2(q)=0$.  Consider  the inclusion $L^{\otimes 2}\subset L^{\otimes 2}(q)$, then  we can fix now an ordered  basis of $M$ given by $\{\sigma s^2_1,\sigma s_1s_2,\sigma s ^2_2,t\}$  where $\sigma \in H^0(\mathcal O_C(q))$, $\sigma \neq 0$, and  $t(q)\neq 0.$

Consider the rational functions $$x = \frac{s_1}{s_2},   \ g= \frac{\sigma s_2^2}{t},  \  y = \frac{1}{g}= \frac{ t}{\sigma s_2^2}. $$

\begin{remark}
\label{remmaroni}
Assume there exists a deformation $\zeta$ of rank $1 = Cliff(C) = gon(C)-2$, that is asymptotic and it is not a Schiffer variation. Let $L$ be the $g^1_3$. Then we can assume $\zeta = [ \debar(\frac{\rho_D}{\omega})]$, where $D$ is the zero divisor of $s_1$ and $\omega = s_1t$. 
\end{remark}
\begin{proof}

Since  $\zeta$ is not a Schiffer variation and $h^0(L^{\otimes 2}) =3$,  by Theorem \ref{cliff=d} we know that $\zeta$ is split and by Proposition \ref{prop_split}, so  we can write $\omega = s_1 \tau$, where $\tau \in H^0(K_C\otimes L^{\vee})$ and $s_1 \in H^0(L)$ have disjoint support. Thus $\tau$ does not belong to $\langle \sigma s_1^2, \sigma s_1 s_2\rangle$. Moreover if $\tau = \sigma s_2^2$, then $\zeta$ is double split, so,  by Theorem \ref{cliff=d}, it is not asymptotic. Hence we can take  $\omega = s_1 t$. 
\end{proof}

Assume $D = div(s_1) = p_1 + p_2 + p_3$ is reduced. 

\begin{theorem}
With the above notation, let $\zeta = [ \debar(\frac{\rho_D}{\omega})]$ be a deformation of rank 1 that is not a Schiffer variation. Then $\zeta$ is asymptotic if and only if the following conditions are satisfied: 
\begin{equation} 
\label{g}
g(p_1)+g(p_2)+g(p_3)=0
\end{equation} 
\begin{equation} 
\label{g^2}
g^2(p_1)+g^2(p_2)+g^2(p_3)=0
\end{equation} 
\begin{equation} 
\label{eqdg}
Res_{p_1}\frac{dg}{x}+Res_{p_2}\frac{dg}{x}+Res_{p_3}\frac{dg}{x}=0.
\end{equation} 
\end{theorem}

\begin{proof}

The space $W$ of quadrics that contain the canonical curve is spanned by the rank $\leq 4$ quadrics corresponding to all the pencils of $H^0(M)$ (see Remark \ref{triqua}). 
Using the chosen basis and computing the second fundamental form as in Proposition \ref{w}, we easily see  $II(Q) ( \zeta \odot \zeta)=0$, for all the quadrics constructed as above, except for (possibly) the ones corresponding to the three pencils in the subspace $\langle \sigma s_2^2, \sigma s_1 s_2, t\rangle \subset H^0(M)$. Denote by 
$$\Gamma_1 = ts_1\odot \sigma s_2^3  - \sigma s_1s_2^2 \odot ts_2 $$ the quadric corresponding to the pencil $\langle \sigma s_2^2, t \rangle$, 
$$\Gamma_2= ts_1\odot \sigma s_1s_2^2  -  \sigma s_1^2s_2\odot ts_2,$$
corresponding to the pencil $\langle \sigma s_1s_2, t \rangle$
and 
$$ \Gamma_3 = \sigma s_1^2s_2\odot \sigma s_2^3-(\sigma s_1s_2^2)^{\odot 2},$$
corresponding to the pencil $\langle \sigma s_2^2, \sigma s_1s_2\rangle$.


Then, using Proposition \ref{w}, we obtain: 
$$II(\Gamma_3)(\zeta \odot \zeta)=-2\pi i\sum_{i} Res_{p_i}( \frac{\sigma s_1^2s_2}{s_1t}d  (\frac{\sigma s_2^3}{s_1t}) ) + 2\pi i \sum_{i} Res_{p_i}( \frac{\sigma s_1s_2^2}{s_1t}d  (\frac{\sigma s_1 s_2^2}{s_1t}) ) =$$
$$=-2\pi i\sum_{i}Res_{p_i}(\frac{\sigma s_1s_2}{t}d (\frac{\sigma s_2^3}{s_1t}))=$$
$$=-2\pi i\sum_{i}Res_{p_i}(xg d (\frac{g}{x}))= 2\pi i\sum_{i}Res_{p_i}(g^2 \frac{dx}{x}) = 2\pi i\sum  g^2(p_i),$$
$$II(\Gamma_2)(\zeta \odot \zeta)= -2\pi i\sum_{i} Res_{p_i}( \frac{ts_1}{s_1t}d  (\frac{\sigma s_1 s_2^2}{s_1t}) ) + 2\pi i \sum_{i} Res_{p_i}( \frac{\sigma s_1^2s_2}{s_1t}d  (\frac{t s_2 }{s_1t}) )=$$

$$= 2\pi i \sum_{i} Res_{p_i}( xgd  (\frac{1 }{x}) ) = - 2\pi i \sum_{i} Res_{p_i}( g(\frac{dx }{x}) )= -2\pi i\sum  g(p_i),$$

$$II(\Gamma_1)(\zeta \odot \zeta) = -2\pi i\sum_{i} Res_{p_i}( \frac{ts_1}{s_1t}d  (\frac{\sigma s_2^3}{s_1t}) ) + 2\pi i \sum_{i} Res_{p_i}( \frac{\sigma s_1s_2^2}{s_1t}d  (\frac{t s_2 }{s_1t}) )= $$
$$=-2\pi i\sum_{i} Res_{p_i}( d  (\frac{g}{x}) ) + 2\pi i \sum_{i} Res_{p_i}(g d ( \frac{1}{x}) ) =-2\pi i \sum_i Res_{p_i}(\frac{dg}{x}) . $$
Hence $\zeta$ is asymptotic if and only if  equations \eqref{g}, \eqref{g^2}, \eqref{eqdg} are satisfied. 

\end{proof}

\begin{remark}
If $D = div(s_1)$ is not reduced, the computation is the same, one only has to consider the multiplicity of the points $p_i$ and take the sum over the support of $D$. 
\end{remark}

Now, by Riemann Roch,  $h^0(M^{\otimes 3}\otimes L)=24-5=19.$  Note that
$M^{\otimes 3}\otimes L\cong L^{\otimes 7}(3q).$
 Set  $V=H^0(L)$, then $s_1^i\cdot s_2^{n-i}$ is a basis of $Sym^n(V)$ the space of symmetric tensors of $V.$
 Consider the space
  $$t^3\cdot V+t^2\cdot Sym^3(V) \sigma +t\cdot Sym^5(V) \sigma^2 +Sym^7(V) \sigma^3\subset H^0(M^{\otimes 3}\otimes L),$$ in fact
 $ t^{3-i}Sym^{2i+1}(V)\subset H^0(M^{\otimes 3}\otimes L(-iq))$. Since $\dim Sym^k(V)=k+1$,   and $h^0(M^{\otimes 3}\otimes L)=19$, counting dimensions, we  see that  we must have  an equation
 $$\phi_1(s_1,s_2)\cdot t^3+\sigma \phi_3(s_1,s_2)t^2+\sigma^2 \phi_5(s_1,s_2)t+\sigma^3\phi_7(s_1,s_2)=0.$$

Since the curve is trigonal (non hyperelliptic) $\phi_1(s_1,s_2)\neq 0$ and since $\sigma \phi_3(s_1,s_2)t^2+\sigma^2 \phi_5(s_1,s_2)t+ \sigma^3 \phi_7(s_1,s_2)$ vanishes on $q$
we may take $\phi_1(s_1,s_2)=s_2$ then we get the relation:
\begin{equation}  s_2 t^3+\sigma \phi_3(s_1,s_2)t^2+\sigma^2 \phi_5(s_1,s_2)t+\sigma^3 \phi_7(s_1,s_2)=0.
\end{equation}

\begin{remark}
Looking at the order of vanishing of the terms of the above equation in $q$, since $s_2(q) = \sigma(q) =0$, we see that $\phi_3(s_1, s_2)$ must contain the term $s_1^3$. 

By a suitable change of coordinates of the form  $t' := t + \alpha \sigma s_1^2 + \beta \sigma s_1 s_2$, we can assume that $\phi_3(s_1, s_2) = a s_1^3 + b s_2^3$, with $a \neq 0$. 
\end{remark}

\bigskip

Dividing by $ \sigma^3 s_2^7$, we get the equation: 

$$y^3 + \frac{y^2}{s_2^3} \phi_3(s_1, s_2) + \frac{y}{s_2^5} \phi_5(s_1, s_2) + \frac{1}{s_2^7} \phi_7(s_1, s_2)= y^3 + y^2 \psi_3(x) + y\psi_5(x) +  \psi_7(x)=$$
$$= y^3 + y^2(a x^3 + b)  + y\psi_5(x) +  \psi_7(x),$$
where $\psi_k(x)$ is the polynomial in $x$ obtained by $\phi_k(s_1, s_2)$ dividing by $s_2^k$ and $a \neq 0$. 

So we have proven the following

\begin{proposition}
\label{maroni6}
Trigonal curves of genus 6 with Maroni degree $k =1$ are described by the following equation: 
\begin{equation}
\label{maroni6eq}
y^3 + y^2(a x^3 + b)  + y\psi_5(x) +  \psi_7(x) =0,
\end{equation}
where $a, b \in {\mathbb C}$, $ a \neq 0$, $\psi_5, \psi_7$ polynomials  of degree $\leq 5, 7$. 
\end{proposition}

We will now describe the locus of trigonal curves of genus 6 with Maroni degree $k =1$ admitting an asymptotic direction of rank 1 different from a Schiffer variation at a ramification point of the $g^1_3$. 
We have the following 

\begin{theorem}
\label{asy6}
Trigonal curves of genus 6 with Maroni degree $k =1$ admitting an asymptotic direction of rank 1 different from a Schiffer variation at a ramification point of the $g^1_3$ satisfy the following equation: 
\begin{equation}
\label{asy6eq}
y^3 +  y^2 x^3  + yx^2\psi_3(x) +  \psi_7(x)=0,
\end{equation}

where $\psi_3, \psi_7$ are polynomials  of degree $\leq 3, 7$ such that $\psi_7(0) \neq 0$. 
\end{theorem}
\begin{proof}

Recalling that $y = \frac{1}{g}$, equations \eqref{g} and  \eqref{g^2} easily give: $\sum_{i=1}^3 y(p_i) = 0$, $\sum_{i=1}^3 (y(p_i))^2 = 0$, and $ \sum_{i <j} y(p_i )y(p_j) =0$.

Set $y_i := y(p_i)$, since $x(p_i) = 0$,  equation \eqref{maroni6eq} gives
$$y_i^3 + b y_i^2  + y_i\psi_5(0) +  \psi_7(0)=0,  \ \forall i=1,2,3.$$

Hence the equation 

$z^3  + b z^2 + \psi_5(0) z  + \psi_7(0)$ has  the elements $y_i$ as  roots, so $b = -\sum_{i=1}^3 y_i =0$ and $\psi_5(0) = \sum_{i <j} y_i y_j =0$. 

Moreover, since $a \neq 0$, changing $t$ (hence $y$) by a non zero multiple, we can assume that $a=1$. 

So the equation \eqref{maroni6eq} becomes 

$$ P(x, y) = y^3 + y^2 x^3  + yx\psi_4(x) +  \psi_7(x)=0. $$

Notice that $y_i^3 = - \psi_7(0) \neq 0 $, since $s_1$ and $t$ have no common zeros. 

Consider now equation \eqref{eqdg}. We have $\frac{dg}{x}  =\frac{1}{x} d (\frac{1}{y}) =  - \frac{1}{x} \frac{dy}{y^2}$, and $P_x dx + P_ydy=0,$ so $dy = - \frac{P_xdx}{P_y }$. Hence  
$$P_x = 3x^2 y^2 + \psi_4(x) y + xy \psi'_4(x) + \psi'_7(x), \ P_y = 3y^2 + 2y x^3 + x \psi_4(x).$$

 So equation \eqref{eqdg} is
$$0= \sum_{i=1}^3 Res_{p_i} (\frac{dg}{x}) = - \sum_{i=1}^3  Res_{p_i}( \frac {P_x }{y^2 P_y} \frac{dx}{x}) =  - \sum_{i=1}^3  \frac {P_x }{y^2 P_y}(p_i)= - \sum_{i=1}^3  \frac{ \psi_4(0)y_i + \psi'_7(0)}{3y_i^4} =$$
$$ = -\frac{1}{3} \sum_{i=1}^3 \frac{\psi_4(0)}{y_i^3}  -\frac{1}{3} \sum_{i=1}^3 \frac{\psi'_7(0)}{y_i^4}=
   \frac{ \psi_4(0)}{\psi_7(0)} - \frac{1}{3}\frac{ \psi'_7(0) }{\psi_7(0)}\sum_{i=1}^3\frac{1}{y_i} =  \frac{\psi_4(0)}{\psi_7(0)}  ,$$
   hence $\psi_4(0) = 0$ and the equation \eqref{maroni6eq} is 

$$P(x, y) = y^3 +  y^2 x^3  + yx^2\psi_3(x) +  \psi_7(x)=0. $$

\end{proof}

\begin{remark}
Since  in equation \eqref{asy6eq}, we have $\psi_7(0)\neq 0$, changing  $s_1$ and $s_2$ (hence $x$ and $y$) by suitable non-zero multiples, we can assume $\psi_7(0) =1$. 
Hence equation \eqref{asy6eq} depends on 11 parameters, while the dimension of the locus of trigonal curves of genus 6 with Maroni degree 1 has dimension 12. 
\end{remark}

By \cite[Lemma 3.3]{bpz},  the tangent space to the trigonal locus is: 
\begin{equation}
\label{tantri}
T_{{tri}, [C]} = \{ \zeta \in H^1(T_C) \ | \ \zeta \cdot \Omega=0 \in H^1(L^{\otimes 2})\},
\end{equation}
where $\Omega = \mu_{1,L}(s_1 \wedge s_2) = s_2^2d(\frac{s_1}{s_2})$, while the tangent space to the locus of trigonal curves with Maroni degree $k=1$ is 

$$T_{{Maroni}, [C]}=  \{ \zeta \in T_{{tri}, [C]} \ | \ \zeta \cdot \frac{\mu_{1,M}(t \wedge \sigma s_2^2) }{s_2} =0\}.$$
Notice that  $\mu_{1,M}(t \wedge \sigma s_2^2) = t^2  d(g) $ and clearly $s_2$ divides $ t^2  d(g)$, hence the element $\frac{\mu_{1,M}(t \wedge \sigma s_2^2) }{s_2} \in H^0(K_C^{\otimes 2})$. 

Assume $\zeta = [ \debar(\frac{\rho_D}{\omega})]$, with $\omega = s_1 t$ as above. We have the following

\begin{proposition}
\begin{enumerate}
\item $\zeta \in T_{{tri}, [C]}$ if and only if equation \eqref{g} is satisfied.
\item $\zeta \in T_{{Maroni}, [C]}$ and it is asymptotic if and only if it satisfies equation \eqref{asy6eq}
with  $\psi_7(0) \neq 0$, $\psi'_7(0) =0$. 

\end{enumerate}
\end{proposition}
\begin{proof}
Notice that $\zeta \cdot \Omega=0 \in H^1(L^{\otimes 2})$ if and only if $ \zeta \cdot \Omega \cdot s_i=0$, for $i=1,2$. We have 
$$  \zeta \cdot \Omega \cdot s_1 = \int_C \Omega s_1 \frac{ \overline{\partial} \rho_D}{s_1 t} =0,$$
since $\rho_D \equiv 0$ on the zero locus of $t$, while 
$$  \zeta \cdot \Omega \cdot s_2 = \int_C \Omega s_2 \frac{ \overline{\partial} \rho}{s_1 t} =
 \int_C s_2^2 d(\frac{s_1}{s_2} )   s_2 \frac{ \overline{\partial} \rho}{s_1 t}=  \int_C\frac{ s_2^2}{t} \overline{\partial} \rho \frac{dx}{x} = \sum_{i =1}^3 g(p_i),$$
 hence the first statement follows.

We have 
$$\int_C  \frac{\mu_{1,M}(t \wedge \sigma s_2^2) }{s_2} \wedge \frac{ \overline{\partial} \rho}{s_1 t} = \int_C  \frac{t^2 }{s_2} d (\frac{\sigma s_2^2}{t})\wedge \frac{ \overline{\partial} \rho}{s_1 t} =$$
$$\int_C  d (\frac{\sigma s_2^2}{t})\wedge \frac{ \overline{\partial}( \rho t)}{s_1 s_2} = -\int_C d (\frac{\rho t}{s_1 s_2} d (\frac{\sigma s_2^2}{t})) = \int_{\cup \partial D_i} \frac{t}{s_1s_2} d (\frac{\sigma s_2^2}{t})=$$
$$= \sum_{i=1}^3 Res_{p_i}( \frac{t}{s_1s_2}dg)=  \sum_{i=1}^3 Res_{p_i} (\frac{f}{g}dg).$$

Now we have: 

$\frac{f}{g}dg= \frac{y}{x} d( \frac{1}{y})= - \frac{1}{xy} dy,$ $dy = - \frac{P_xdx}{P_y}, $ so 
$$\sum_{i=1}^3 Res_{p_i} (\frac{f}{g}dg)= -\sum_{i=1}^3 Res_{p_i}(\frac{P_x}{y P_y}) \frac{dx}{x} = - \sum_{i=1}^3\frac{P_x}{y P_y} (p_i) = \sum_{i=1}^3 \frac{ \psi'_7(0)}{3 y_i^3} =   \frac{ \psi'_7(0)}{\psi_7(0)}. $$

So $\zeta$  asymptotic is  in $T_{{Maroni}, [C]}$  if and only if $\psi'_7(0) =0$. 

 \end{proof}

\subsection{Genus 7 Maroni special}

Assume $C$ is a (non hyperelliptic) trigonal curve of genus 7 with Maroni degree $k =1$. 
We will show that  also in this case there can exist asymptotic directions that are not Schiffer variations in the ramification points of the $g^1_3$ and we will describe them. Moreover we give a parametrisation of the locus of the trigonal curves of genus 7 with Maroni degree 1 giving an explicit equation and we  describe  the sublocus of those trigonal curves admitting such asymptotic directions.

Take a deformation $\zeta$ of rank $1 = Cliff(C) = gon(C)-2$ and let $L$ be the $g^1_3$. The assumption on the Maroni degree gives $h^0(L^{\otimes 3}) =5$. By Riemann Roch we have $h^0(K_C \otimes L^{-3}) =2$, so $K_C \otimes L^{-3}$ is the $g^1_3$, so $K_C = L^{\otimes 4}$. Set as usual $V= H^0(L) = \langle s_1, s_2 \rangle$, then $H^0(L^{\otimes 3}) =\langle s_1^3, s_2^3, s_1^2s_2, s_1s_2^2, t \rangle$. 

As in the case $g=6$, we can use Remark \ref{remmaroni} and assume $\zeta = [ \debar(\frac{\rho_D}{\omega})]$, where   $D$ is the zero divisor of $s_1$ and  $\omega = s_1t$.


Set $x= \frac{s_1}{s_2}$, $g = \frac{s_2^3}{t}$, $y = \frac{1}{g} = \frac{t}{s_2^3}$, $D = div(s_1) = p_1 + p_2 +p_3$.  In the case $D$  not reduced, the computation is the same, one only has to consider the multiplicity of the points $p_i$ and take the sum over the support of $D$.

\begin{theorem}
With the above notation, let $\zeta = [ \debar(\frac{\rho_D}{\omega})]$ be a deformation of rank 1  that is not a Schiffer variation. Then $\zeta$ is asymptotic if and only if equations  \eqref{g}, \eqref{g^2}, \eqref{eqdg} are satisfied, where $x= \frac{s_1}{s_2}$, $g = \frac{s_2^3}{t}$. 
\end{theorem}

\begin{proof}

A basis of $I_2(K_C)$ is given by the rank $\leq 4 $ quadrics determined by the pencils in $H^0(K_C \otimes L^{\vee}) = H^0(L^{\otimes 3})= \langle s_1^3, s_2^3, s_1^2s_2, s_1s_2^2, t \rangle$. Computing the second fundamental form as in Proposition \ref{w}, we easily see  $II(Q) ( \zeta \odot \zeta)=0$ for all the quadrics constructed as above, except for (possibly) the three pencils in the subspace $\langle  s_2^3, s_1 s_2^2, t\rangle \subset H^0(K_C \otimes L^{\vee})$.

Denote by  $$\Gamma_1 = ts_1\odot  s_2^4  -ts_2\odot s_1s_2^{3}, \ \Gamma_2= ts_1\odot  s_1s_2^3  -ts_2\odot  s_1^2s_2^2, \ \Gamma_3 =  s_1^2s_2^2\odot  s_2^4-( s_1s_2^3)^{\odot 2}$$ the corresponding three quadrics.

A similar computation as in the case $g=6$ shows that  $\zeta$ is asymptotic if and only if equations \eqref{g}, \eqref{g^2}, \eqref{eqdg} are satisfied. 
\end{proof}
Observe that 

\begin{equation}
t^3 + t^2 Sym^3(V) + t Sym^6(V) + Sym^9(V) \subset H^0(L^{\otimes 9})=H^0(M^{\otimes 3}).
\end{equation}
Since $\dim(Sym^k(V) )= k+1$ and $h^0(L^{\otimes 9}) = 21$, we have an equation 
\begin{equation}
t^3 + t^2 \phi_3(s_1, s_2) + t \phi_6(s_1, s_2) + \phi_9(s_1, s_2) =0.  
\end{equation}
Dividing by $s_2^9$, we get the equation:

\begin{equation}
P(x,y) =y^3 + y^2 \psi_3(x) + y \psi_6(x) + \psi_9(x) =0.
\end{equation}

So we have proven the following 
\begin{proposition}
\label{maroni7}
Trigonal curves of genus 7 with Maroni degree $k =1$ are described by the following equation: 
\begin{equation}
\label{maroni7eq}
y^3 + y^2 \psi_3(x) + y \psi_6(x) + \psi_9(x) =0,\end{equation}
where  $\psi_3, \psi_6, \psi_9$ are polynomials  of degree $\leq 3, 6, 9$. 
\end{proposition}

We will now describe the locus of trigonal curves of genus 7 with Maroni degree $k =1$ admitting an asymptotic direction of rank 1 different from a Schiffer variation at a ramification point of the $g^1_3$.

\begin{theorem}
\label{asy7}
Trigonal curves of genus 7 with Maroni degree $k =1$ admitting an asymptotic direction of rank 1 different from a Schiffer variation at a ramification point of the $g^1_3$ satisfy the following equation: 
\begin{equation}
\label{asy7eq}
y^3 + y^2 x \psi_2(x) + y x^2 \psi_4(x) + \psi_9(x),
\end{equation}
where $  \psi_2, \psi_4, \psi_9$ are polynomials of degree $\leq 2,4,9$ and $\psi_9(0) \neq 0$. 

\end{theorem}
\begin{proof}
Set $y_i := y(p_i)$, then equations \eqref{g}, \eqref{g^2} imply that $\sum_{i=1}^3 y_i = 0$, $\sum_{i=1}^3 y^2_i = 0$, $\sum_{i<j} y_i y_j =0$. Since $x(p_i) =0$, $\forall i =1,2,3$, equation \eqref{maroni7eq} gives 
$$ y_i^3 + y_i^2 \psi_3(0) + y_i \psi_6(0) + \psi_9(0) =0, \ \forall i =1,2,3.$$
Hence the equation 
$$z^3 + z^2 \psi_3(0) + z \psi_6(0) + \psi_9(0) =0,$$
has $y_1, y_2, y_3$ as roots. 
So we obtain 
$$\psi_3(0) = - \sum_{i=1}^3 y_i = 0, \ \psi_6(0) = \sum_{i<j} y_i y_j= 0.$$

Hence  equation \eqref{maroni7eq} is of the form 

$$P(x,y) = y^3 + y^2 x \psi_2(x) + y x \psi_5(x) + \psi_9(x),$$

and $\psi_9(0) \neq 0$, since $s_1$ and $t$ do not have common zeros. Now we compute $dg = -\frac{dy}{y^2}$ and $P_x dx + P_y dy =0$, so $dy = - \frac{P_x dx}{P_y}$. 

So equation \eqref{eqdg} becomes 
$$0 = \sum_{i=1}^3  \frac{P_x}{y^2P_y}(p_i) = \sum_{i=1}^3 \frac{y_i^2 \psi_2(0) + y_i \psi_5(0) + \psi'_9(0)}{3 y_i^4} = $$
$$= \frac{1}{3}\sum_{i=1}^3 \frac{ \psi_2(0) }{ y_i^2} + \frac{1}{3}\sum_{i=1}^3 \frac{ \psi_5(0) }{ y_i^3} + \frac{1}{3}\sum_{i=1}^3 \frac{ \psi'_9(0) }{ y_i^4}
= -\frac{\psi_5(0)}{\psi_9(0)},$$
since $\sum_{i=1}^3 \frac{1}{y_i} = \sum_{i=1}^3 \frac{1}{y^2_i} =0$ and $y_i^3 = -\psi_9(0).$ So we get $\psi_5(0) = 0$, and equation \eqref{maroni7eq} is 
$$P(x,y) = y^3 + y^2 x \psi_2(x) + y x^2 \psi_4(x) + \psi_9(x)=0$$
where $\psi_9(0) \neq 0$. 
\end{proof}

\begin{remark}
If $\zeta$ is an asymptotic direction, then it is tangent to the trigonal locus. 
 \end{remark}
 \begin{proof}
By \eqref{tantri}, $\zeta \in T_{{tri}, [C]}$ if and only if $\zeta \cdot \Omega \cdot s_i s_j= 0, \ i,j=1,2$. 

We have: 

$$  \zeta \cdot \Omega \cdot s_1 s_j = \int_C \Omega s_1 s_j \frac{ \overline{\partial} \rho}{s_1 t} = 0, $$
for $j=1,2$, 
 
$$  \zeta \cdot \Omega \cdot s^2_2 = \int_C \Omega s_2^2 \frac{ \overline{\partial} \rho}{s_1 t} =
 \int_C s_2^2 d(\frac{s_1}{s_2} )   s^2_2 \frac{ \overline{\partial} \rho}{s_1 t}=  \int_C\frac{ s_2^3}{t} \overline{\partial} \rho \frac{dx}{x} = \sum_{i =1}^3 g(p_i)=0,$$
if and only if it satisfies equation \eqref{g}.

\end{proof}
\subsection{Genus 5}

Assume $C$ is a non hyperelliptic  trigonal curve of genus 5, denote by $L$ the trigonal line bundle. 

We will show that  also in this case there can exist asymptotic directions that are not Schiffer variations in the ramification points of the $g^1_3$ and we will describe them. Moreover we give a parametrisation of the locus of the trigonal curves of genus 5  and of  the sublocus of those trigonal curves admitting such asymptotic directions.

Then, $K_C \otimes L^{-2}$ has degree 2 and by Riemann Roch $h^0(K_C \otimes L^{-2}) = 1$, so $K_C=L^{\otimes 2}\otimes \cO_C(p+q)  $ for some $p,q \in C$. Then  $M= K_C \otimes L^{\vee} = L(p + q)$.  Set $H^0(L) = \langle s_1, s_2 \rangle$,   $H^0(K_C \otimes L^{\vee}) = \langle \sigma s_1, \sigma s_2, t\rangle$, where $div(\sigma) = p+q$ and $t(p) \neq 0$, $t(q) \neq 0$, since $C$ is not hyperelliptic. 
So we have $H^0(K_C) = H^0(L^{\otimes 2}(p+q)) = \langle \sigma s_1^2, \sigma s_1s_2, \sigma s_2^2, ts_1, t s_2\rangle.$

Set $x = \frac{s_1}{s_2}$, $g = \frac{\sigma s_2}{t}$, $y = \frac{1}{g}$, $D = div(s_1) = p_1 + p_2 +p_3$, and again we can choose $\omega = s_1t$.  In the case $D$  not reduced, the computation is the same, considering the multiplicity of the points $p_i$ and taking the sum over the support of $D$.

We have the following 

\begin{theorem}
With the above notation, let $\zeta = [ \debar(\frac{\rho_D}{\omega})]$ be a deformation of rank 1 that is not a Schiffer variation Then $\zeta$ is asymptotic if and only if equations  \eqref{g}, \eqref{g^2}, \eqref{eqdg} are satisfied, where $x = \frac{s_1}{s_2}$, $g = \frac{\sigma s_2}{t}$. 
\end{theorem}
\begin{proof}
The space $I_2(K_C)$ is three dimensional and it is generated by the following quadrics: 
$$\Gamma_1 = \sigma s_2^2 \odot s_1t - \sigma s_1 s_2 \odot   s_2t,  \  \Gamma_2 = \sigma s_1^2 \odot s_2t - \sigma s_1 s_2 \odot   s_1t, \ \Gamma_3 = \sigma s_1^2 \odot \sigma s_2^2 - \sigma s_1 s_2 \odot  \sigma s_1 s_2,$$
corresponding to the three pencils of $|K \otimes L^{\vee}|$. 

A straightforward  computation as in the case $g=6,7$ shows that $\zeta$ is asymptotic if and only  equations \eqref{g}, \eqref{g^2}, \eqref{eqdg} are satisfied. 

\end{proof}

Now, by Riemann Roch,  $h^0(M^{\otimes 3}\otimes L^{\otimes 2})= h^0(L^{\otimes 5}(3p +3q))=21-4=17.$  
 Set  $V=H^0(L)$ and  consider 
 $$t^3\cdot Sym^2V+t^2\cdot Sym^3(V) \sigma +t\cdot Sym^4(V) \sigma^2 +Sym^5(V) \sigma^3\subset H^0(M^{\otimes 3}\otimes L^{\otimes 2}).$$  
 Counting dimensions, we see that we have  an equation
 $$\phi_2(s_1,s_2)\cdot t^3+\sigma \phi_3(s_1,s_2)t^2+\sigma^2 \phi_4(s_1,s_2)t+\sigma^3\phi_5(s_1,s_2)=0. $$

Since the curve is not hyperelliptic,  $\phi_2(s_1,s_2)\neq 0$, moreover,  $\sigma \phi_3(s_1,s_2)t^2+\sigma^2 \phi_4(s_1,s_2)+ \sigma^3 \phi_5(s_1,s_2)$ vanishes on $p+q$. Thus we can choose $s_2$ such that $(s_1+s_2) (p)=0$, $(s_1+s_2) (q)\neq 0$, $(s_1-s_2) (p)\neq 0$, $(s_1-s_2) (q)=0$ and $\phi_2(s_1,s_2)=s_1^2 -s_2^2$. 



Dividing the equation by $\sigma^3 s_2^5$,  we get 
$$ y^3(x^2-1) +y^2 \psi_3(x) +y \psi_4(x) +\psi_5(x)=0.$$

So we have proven the following 
\begin{proposition}
\label{maroni5}
Trigonal curves of genus 5 are described by the following equation: 
\begin{equation}
\label{maroni5eq}
y^3(x^2-1) +y^2 \psi_3(x) +y \psi_4(x) +\psi_5(x)=0,\end{equation}
where  $\psi_3, \psi_4, \psi_5$ are polynomials  of degree $\leq 3, 4, 5$. 
\end{proposition}

 We will now describe the locus of trigonal curves of genus 5  admitting an asymptotic direction of rank 1 different from a Schiffer variation at a ramification point of the $g^1_3$. 

\begin{theorem}
\label{asy5}
Trigonal curves of genus 5  admitting an asymptotic direction of rank 1 different from a Schiffer variation at a ramification point of the $g^1_3$ satisfy the following equation: 
\begin{equation}
\label{asy5eq}
y^3(x^2-1) + y^2x\alpha_2(x) + y x^2 \chi_2(x) + \psi_5(x),
\end{equation}
where $  \alpha_2, \chi_2$ are polynomials of degree $\leq 2$, $\psi_5$ is a polynomial of degree $\leq 5$ and $\psi_5(0)=1$. 

\end{theorem}
\begin{proof}

Equations \eqref{g}, \eqref{g^2} are equivalent to $\sum_i y_i =0$, $\sum_i y^2_i =0$, $\sum_{i,j} y_i y_j=0,$ where $y_i := y(p_i)$. 
Since $x(p_i) = 0$,   equation \eqref{maroni5eq} gives 

$$-y_i^3 +  y_i^2 \psi_3(0) + y_i\psi_4(0) +  \psi_5(0)=0,  \ \forall i=1,2,3.$$

Hence the polynomial $-z^3  + \psi_3(0) z^2 + \psi_4(0) z  + \psi_5(0)$ has the elements $y_i$ as roots, so $\psi_3(0) = \sum_{i=1}^3 y_i =0$ and $\psi_4(0) = -\sum_{i <j} y_i y_j =0$, while $\psi_5(0) \neq 0$. 

Writing $g$ as a function of $x$, $\frac{dg}{x} = g'(x) \frac{dx}{x}$, so condition \eqref{eqdg} is  $0 = \sum_i g'(p_i) = - \sum_i \frac{y'}{y^2} (p_i)$. 
Since $dy = - \frac{P_x}{P_y} dx$, and 
$$P_x= 2x y^3 + y^2 \psi_3'(x) + y \psi'_4(x) + \psi'_5(x),$$
$$P_y= 3y^2(x^2-1) + 2y \psi_3(x) + \psi_4(x) ,$$

Then 
$$0=\sum_i \frac{y'}{y^2} (p_i) =- \sum_i \frac{y_i^2\psi'_3(0) + y_i \psi'_4(0) + \psi'_5(0)}{3y_i^4}= - \frac{1}{3} \sum_i \frac{\psi'_4(0)}{y_i^3} = - \frac{\psi'_4(0)}{\psi_5(0)}.$$

So, $\psi'_4(0) =0$ and  equation \eqref{maroni5eq} becomes: 

$$P(x,y) = y^3(x^2-1) + y^2x\alpha_2(x) + y x^2 \chi_2(x) + \psi_5(x),  $$

where we can assume $\psi_5(0) = 1$.
\end{proof}

\begin{remark}
Notice that changing $t$ with $t + \lambda \sigma s_1$, the deformation $\zeta$ does not change. Hence we can change $y$ with $y + \lambda x$ in such a way that $\alpha_2(x) = x \alpha_1(x)$, where $\alpha_1$ is a polynomial of degree $\leq 1$. So the equation depends on 10 parameters, while the trigonal locus has dimension 11. 
\end{remark}

\begin{remark}
If $\zeta$ is an asymptotic direction, then it is tangent to the trigonal locus. 
\end{remark}
\begin{proof}


By \eqref{tantri}, the infinitesimal deformation $\zeta = [ \debar(\frac{\rho_D}{\omega})]$  is tangent to the trigonal locus if and only if 

$$ \zeta\cdot \Omega  = \int_C \Omega \frac{ \overline{\partial} \rho}{s_1 t} =
 \int_C s_2^2 d(\frac{s_1}{s_2} )  \wedge   \frac{ \overline{\partial} \rho}{s_1 t}=  \int_C d( d(\frac{s_1}{s_2}) s_2^2 \frac{\rho}{s_1t})= \sum_i Res_{p_i}  d(\frac{s_1}{s_2}) s_2^2 \frac{1}{s_1t} =$$
$$= \sum_i Res_{p_i} g \frac{dx}{x} =  \sum_i g(p_i)=0,$$

which is equation \eqref{g}. 
\end{proof}


  \end{document}